\documentclass[11pt]{article}
\usepackage{amsmath}
\usepackage{amsthm}
\usepackage{amssymb}
\usepackage{latexsym}
\usepackage{lscape}
\usepackage{epsfig}
\usepackage{pstricks}
\usepackage{enumerate}
\usepackage{exscale}
\usepackage{relsize}
\usepackage{amsmath}
\usepackage{indentfirst}
\usepackage{graphicx} 
\usepackage{epstopdf}
\usepackage{lineno}
\usepackage{subcaption}
\usepackage{authblk}	
\usepackage{graphicx}
\usepackage[square,sort,comma,numbers]{natbib}

\newtheorem{theorem}{Theorem}[section]
\newtheorem{lemma}[theorem]{Lemma}
\newtheorem{proposition}[theorem]{Proposition}
\newtheorem{corollary}[theorem]{Corollary}
\theoremstyle{definition}

\theoremstyle{remark}
\newtheorem{remark}{Remark}

\usepackage{mathrsfs,graphicx,latexsym,tikz,color,euscript}
\usepackage{amsfonts,amssymb,amsmath,amscd,amsthm}
\usepackage{epstopdf}
\usepackage{CJKutf8}
\usepackage[colorlinks=true,linkcolor=blue,urlcolor=blue,citecolor=blue,linkbordercolor=white,urlbordercolor=white,citebordercolor=white]{hyperref}
\usepackage{color}

\newcommand{\beq}{\begin{eqnarray*}}
	\newcommand{\eeq}{\end{eqnarray*}}

\theoremstyle{plain}

\theoremstyle{definition}

\begin{document}

\title{Periodic orbits of the Stark problem}

\author{
	Ku-Jung Hsu$^{1,}$\thanks{Partially supported by National Natural Science Foundation of China(No.12101363, No.12071255, No.12171281), Natural Science Foundation of Fujian Province of China (No.2023J01123), Scientiﬁc Research Funds of Huaqiao University(No.22BS101). Email: s9921803@m99.nthu.edu.tw}, Wentian Kuang$^{2,}$\thanks{Partially supported by National Natural Science Foundation of China(No.12271300, No.11901279). E-mail: kuangwt@gbu.edu.cn}\\
	$^{1}${\small School  of  Mathematical  Sciences,  Huaqiao  University,  Quanzhou,  362021,  P.R. China. \\}
	$^{2}${\small School of Sciences, Great Bay University,
	Dongguan, 523000, P.R. China,\\Great Bay Institute for Advanced Study, Dongguan 523000, China.\\}
	}

\date{}

\maketitle

\begin{abstract}
	The Stark problem is Kepler problem with an external constant acceleration. In this paper, we  study the periodic orbits for Stark problem for both planar case and spatial case. We have conducted a detailed analysis of the invariant tori and periodic orbits appearing in the Stark problem, providing a more refined characterization of the properties of the orbits. Interestingly, there exists a family of circular orbits in the spatial case, some of which are quite stable with $L$ being fixed. 
\end{abstract}


\section{Introduction}\label{sec:int}
There are very few integrable systems in Celestial Mechanics. Well-known examples include the Kepler problem, the Euler problem (two-center Newtonian gravitational motion), and the Stark problem. Our focus in this paper is on the Stark problem. It is named after the German physicist Johannes Stark who discovered in 1913 the Stark effect, whose discovery was one of the reasons that Stark received 1919 the Nobel prize.

This problem arises in scenarios like the dynamics of an electron attracted by a proton in a constant external electric field or a rocket attracted by a planet and subject to constant thrust. The Stark problem was shown to be integrable first by Lagrange who reduced it to quadratures at the end of 18th century. In the middle of 19th century, Jacobi found that the Stark system admits the separation of its variables in the parabolic coordinates. In the early 20th century, the Stark problem gained renewed attention with the emergence of quantum mechanics, which was served as a model to explain the Stark effect or to understand the behavior of charged particles in electric fields. 

There has been many research related to the Stark problem, and here we list just a few works that are known to us. In \cite{Beletsky2001}, V.Beletsky gave a description for different types of orbits in planer Stark problem in parabolic coordinates. G. Lantoine and R. P. Russell \cite{LR2011} derived the explicit expressions for different types of solutions, which are in forms of Jacobi elliptic function. U.Frauenfelder studied the problem from another perspective in \cite{Urs2023} and showed that the bounded component of the energy hypersurfaces for energies below the critical value can be interpreted as boundaries of concave toric domains.

To the best of the author's knowledge, most existing conclusions have been derived in parabolic coordinates. In these coordinates, periodic orbits reside on invariant tori with periods that are rational dependent. In this paper, we focus on studying the periodic orbits for Stark problem. Both planar case and spatial case are considered. We have conducted a detailed study of the invariant tori and periodic orbits appearing in the Stark problem. It is noteworthy that the properties of orbits Stark problem exhibit significant distinctions from those of Kepler problem. There are also many differences for orbits in the spatial case and the planar case . Our main results are presented separately for the planar and spatial cases.

For planar case, we establish the existence or non-existence of various periodic orbits with different topological properties in the initial coordinate system. As a corollary, we prove that the minimizer with two boundary points on positive $x$-axis and negative $x$-axis respectively must be a collision orbits on negative axis. Additionally, we investigate the invariant tori in the fixed energy case.

In the spatial case, when the angular momentum is given, the problem is also separable,  as in the planar case.   We provide a concise classification for the orbits of the spatial Stark problem with varying angular momentum and energy. Interestingly, there exists a family of stable circular orbits in spatial case. This indicates that planetary orbits in the universe could be more stable in the presence of an external force field.
\subsection{The planar case}
The equation for planar Stark problem is
	\begin{equation*}
		\left\{\begin{aligned}
			\ddot{x}&=-\frac{\mu x}{r^3}+\epsilon,  \\
			\ddot{y}&=-\frac{\mu y}{r^3}, \\
		\end{aligned}
		\right.
	\end{equation*}
	where $r=\sqrt{x^2+y^2}$.

By rescaling, we can set the constant $\mu$ to be $1$. Simple computation shows that, $q(t)$ is a solution of above equation if and only if $\forall c>0$, $cq(c^{-\frac{3}{2}}t)$ is a solution with $\epsilon$ replaced by $\frac{\epsilon}{c^2}$. Without loss of generality, we can assume $\epsilon=1$ and $\mu=1$ to simplify notations, unless stated otherwise.

\begin{equation}\label{eq:stark}
	\left\{\begin{aligned}
		\ddot{x}&=-\frac{ x}{r^3}+1,  \\
		\ddot{y}&=-\frac{ y}{r^3}, \\
	\end{aligned}
	\right.
\end{equation}

Obviously, there is a unique equilibrium $(1,0)$ for Stark problem (\ref{eq:stark}).  
   Let $q=(x,y)$ be the position of the particle. The equation corresponds to a Hamiltonian system with Hamiltonian function 
	\begin{equation}
		H(q)=K-U=\frac{1}{2}\dot{q}^2-\frac{1}{|q|}- x,
	\end{equation}
	where $K=\frac{1}{2}\dot{q}^2$, $U=\frac{1}{|q|}+ x$. The system is autonomous, thus the Hamiltonian is preserved along solutions.

In this paper,  we focus on investigating periodic orbits in the Stark problem. A \textsl{brake orbit} is a periodic orbit such that there exists some time when $\dot{q}=0$. Since the system is symmetric with respect to the $x$-axis, it is natural to search for symmetric periodic orbits.  We prove existences of various symmetric periodic orbits. Our main results are outlined below. 

\begin{theorem}\label{thm:brake}
       Given $a\in(0,1]$ and  $b\in[0,+\infty)$.		
		Let $q=(x,y)$ be the solution of (\ref{eq:stark}) with initial conditions $(x(0),y(0))=(a,0)$ and $(\dot{x}(0),\dot{y}(0))=(0,b)$.  Then $q$ is clearly a symmetric solution and the following holds:\\
  	(i) $q$ is bounded if and only if $b\le \sqrt{\frac{2}{a}-2a}$;\\
	(ii) when $b = \sqrt{\frac{2}{a}-2a}$,  the solution $q$ (denoted as $q^*_a$) is a brake orbit. \\
	Moreover,  the solution $q^*_a=(x^*_a,y^*_a)$ satisfies the following conditions:\\
  (iii)  $|q^*_a|+x^*_a=2a$ and $H=-2a$;\\
  (iv)  $\dot{q}^*_a(\hat{t})=0$ at the moment $\hat{t}$ with $|q^*_a(\hat{t})|=1$;\\
  (v) there exists a sequence of solutions of (\ref{eq:stark}) that tends to the brake orbit $q^*_a$ as $t \rightarrow \pm\infty$;\\
  (vi) $q^*_a$ is a unique bound solution of (\ref{eq:stark}) with $(x(0),y(0))=(a,0)$ which do not intersects negative $x$-axis.
	\end{theorem}

From the proof, we also have the following conclusion.
\begin{corollary}\label{cor:1}
	Any bounded orbit of Stark problem is  either a collision ejection orbit on the negative $x$-axis or it consistently remains within the unit disc.
\end{corollary}

\begin{figure}
	\centering
		\includegraphics[width=0.5\linewidth]{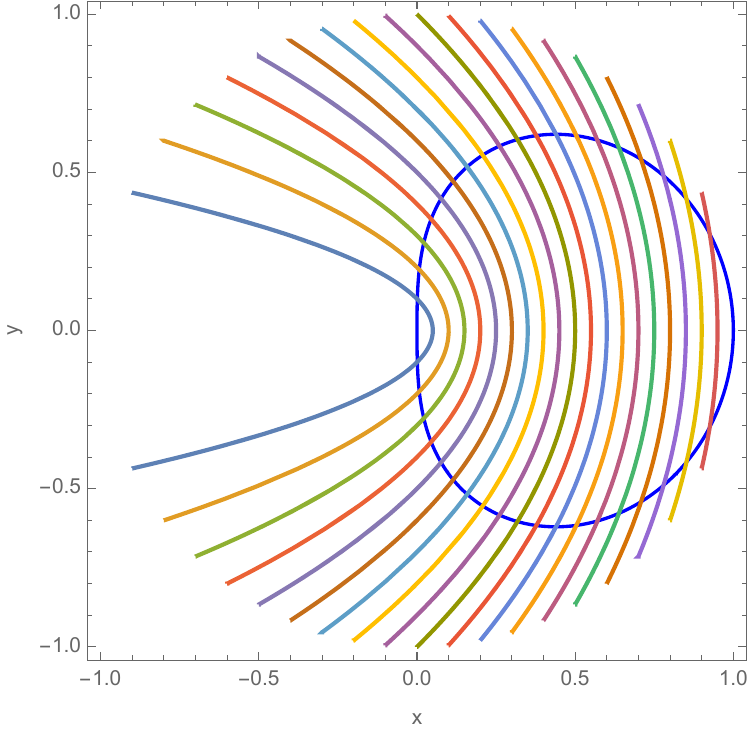}
	\caption{Brake orbits proved in Theorem~\ref{thm:brake}(i). The closed blue curve is defined by $\ddot{x}=0$.}
	\label{fig:brake}
\end{figure}

A sequence of brake orbits is illustrated in Fig.\ref{fig:brake}. Note that the brake orbits in Theorem~\ref{thm:brake} can be achieved by two ways: launching from $(0,1)$ in $x$-axis or descending from a point on unit circle. By perturbation under the two settings, we can find numerous periodic solutions with different topological characteristics. 

\begin{theorem}\label{thm:Ak}
	For any $a\in (0,1)$, there exist $k_p\in \mathbb{N}$, and $\{b_k\}_{k\ge k_p}, \{b'_k\}_{k \ge k_p}$ such that two sequences of orbits $\{A_k\}_{k\ge k_p}$ and $\{B_k\}_{k\ge k_p}$ with the same initial positions $(x(0),y(0))=(a,0)$, while the initial velocities for $A_k$ and $B_k$ are $(\dot{x}(0),\dot{y}(0))=(0,b_k)$ and $(\dot{x}(0),\dot{y}(0))=(0,b_k')$, respectively,  satisfy the following properties:\\	
	(i) $A_k$ is a periodic orbit which intersects positive $x$-axis $k$ times and then intersects the negative $x$-axis perpendicularly;\\
	 (ii) $B_k$ is a collision orbit which intersects positive $x$-axis $k$ times before collision;\\
	 (iii) $\lim_{k\rightarrow +\infty}b_k=\lim_{k\rightarrow +\infty}b_k'=\sqrt{\frac{2}{a}-2a}$ , $\lim_{a\rightarrow 1}k_p=+\infty$;\\
	 (iv) $\{b_k\}$ and $\{b_k'\}$ satisfies the inequalities
	$b_k<b_k'$ and $b_k'<b_{k+1}$, for all $k\ge k_p$.
\end{theorem}
\begin{theorem}\label{thm:Bk}
	For any $a'\in (-1,1)$,  there exist  $k_b\in \mathbb{N}$ and $\{d_k\}_{k \ge k_b}$, $\{d_k'\}_{k \ge k_b}$ such that  two sequences of orbits $\{A_k'\}_{k\ge k_b}$ and $\{B_k'\}_{k \ge k_b}$ with zero initial velocities $(\dot{x}(0),\dot{y}(0))=(0,0)$, while the initial positions for $A_k$ and $B_k$ are $({x}(0),{y}(0))=(a',d_k)$ and $({x}(0),{y}(0))=(a',d_k')$, respectively. The following holds:\\	
	(i) $A_k'$ is a brake orbit which intersects positive $x$-axis $k$ times and then intersects the negative $x$-axis perpendicularly;\\
	(ii) $B_k'$ is a collision orbit which  intersects positive $x$-axis $k$ times before collision;\\
	(iii) $\lim_{k\rightarrow +\infty}d_k=\lim_{k\rightarrow +\infty}d_k'=\sqrt{1-a'^2}$ , $\lim_{a'\rightarrow 1}k_b=+\infty$;\\
	(iv) $\{d_i\}$ and $\{d_i'\}$ satisfies the inequalities
		$d_k<d_k'$ and $d_k'<d_{k+1}$, for all $k\ge k_b$.
\end{theorem}
\begin{remark}
	Although the orbits in Theorem~\ref{thm:Ak} and Theorem~\ref{thm:Bk} are of different type, they could be on the same invariant torus with different initial phases. Numerically, $k_p$ and $k_b$ can be chosen to be $1$ for $a\in (0,0.9)$. Typical orbits of type $A_1$ and $B_1$ are depicted in Fig.\ref{fig:A1B1}. By the proof of Theorem~\ref{thm:Ak} (resp. Theorem~\ref{thm:Bk}),  we know that large $k_p$ (resp.  large $k_b$) implies that the orbit must oscillate many times along the positive $x$-axis before it hits the negative $x$-axis (resp.  before it collides the origin). This oscillatory phenomenon manifests only when the orbit goes very close to the equilibrium point(the whole orbit will close to $x$-axis), making it challenging to observe numerically. However, it is a consequence of the fact that the equilibrium point is a non-degenerate hyperbolic fixed point. 
\end{remark}

\begin{figure}
	\begin{subfigure}{0.45\textwidth}
		\includegraphics[width=\linewidth]{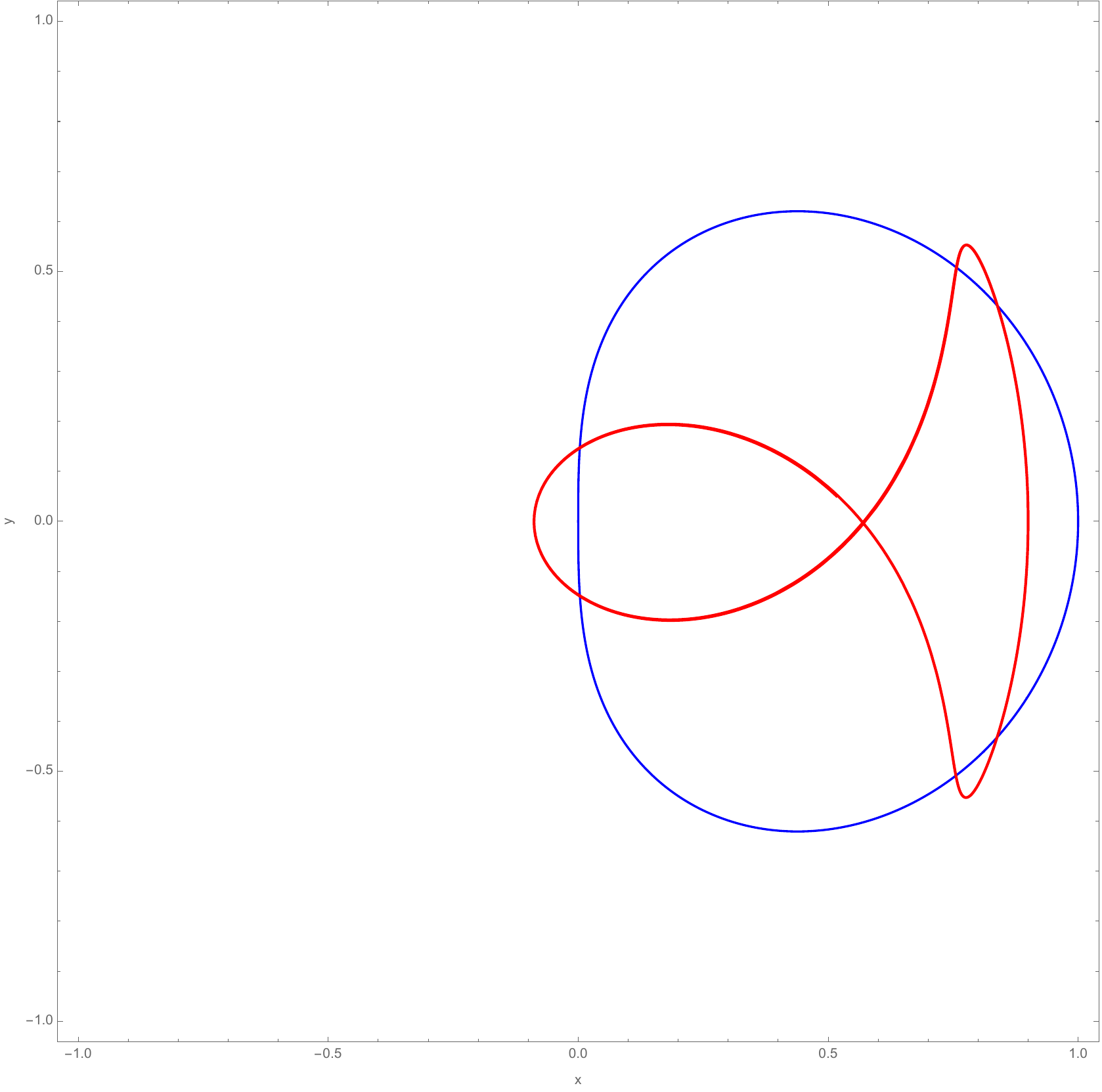}
		\label{fig:A1}
	\end{subfigure}
\hfill
	\begin{subfigure}{0.45\textwidth}
		\includegraphics[width=\linewidth]{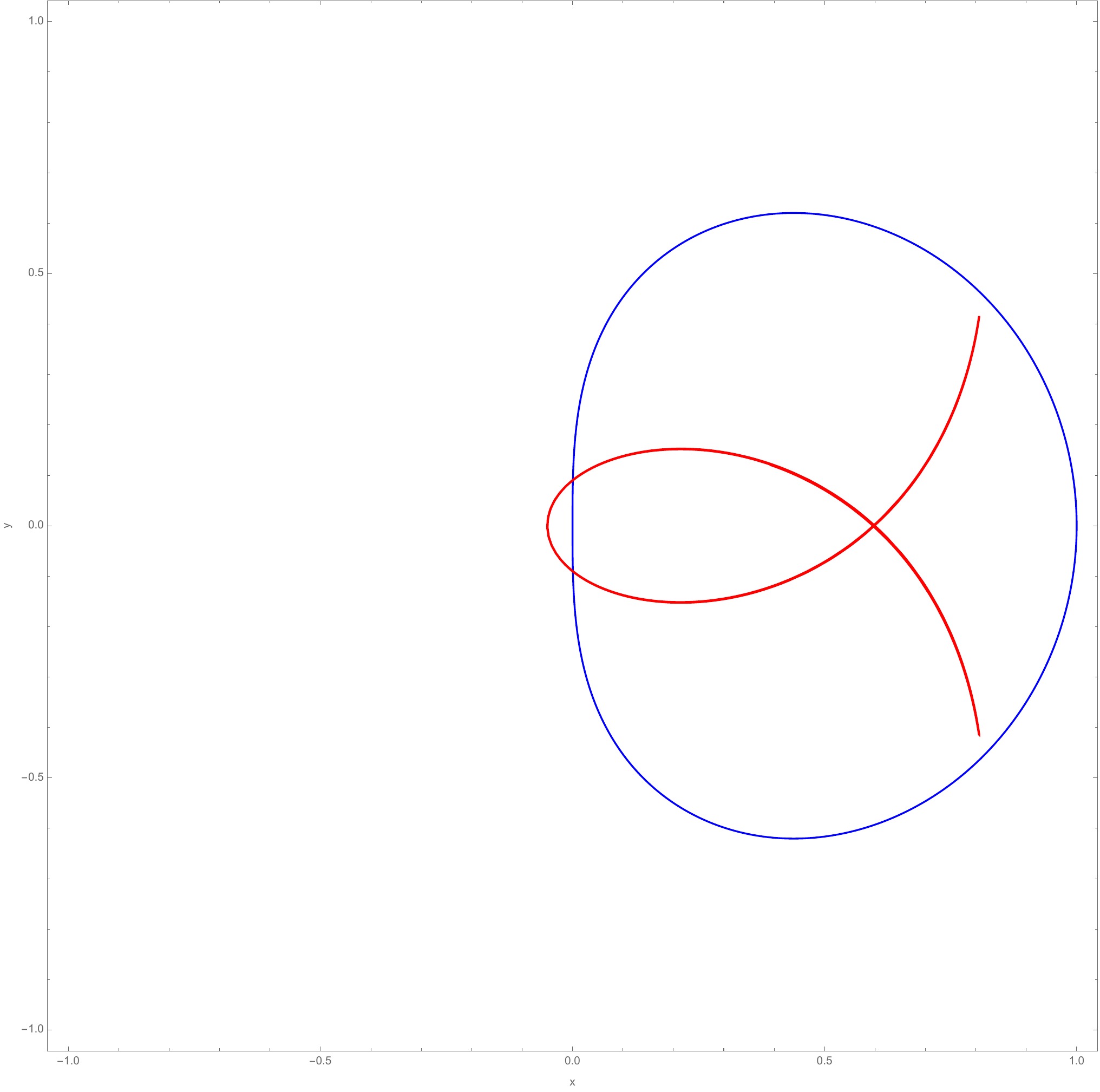}
		\label{fig:B1}
	\end{subfigure}
	
	\caption{Periodic orbits of type $A_1$ and $B_1$.}
	\label{fig:A1B1}
\end{figure}

Systems in Celestial Mechanics have natural variational structure. The equation for Stark problem corresponds to Euler-Lagrange equation for a certain action functional. Certain periodic orbits can be viewed as minimizers within a specific path space. In W.Gordon \cite{Gor1977}, elliptic orbits of the Kepler problem are demonstrated to be action minimizers among the space of loops with a non-zero winding number around the origin. Similar variational characterization results are proved in \cite{Chenkc2017,LongZhang2000,KuangLong2016} for different problems. 

Since the pioneer work of A.Chenciner and R.Montgomery in \cite{CM2000}, variational methods have been employed to establish the existence of various periodic orbits in the $N$-body problem. In \cite{shibiyam2019,ChenKC2022}, the authors proved the existence of different types of periodic orbits in planar $2$-center problem. Consequently, we anticipate proving the existence of periodic orbits for the Stark problem through variational methods.

The equation \eqref{eq:stark} can be seen as the Euler-Lagrangian equation of the following functional 
\begin{equation}\label{eq:lagrange}
	\mathcal{A}(q)=\int \Big(\frac{1}{2}\dot{q}^2+\frac{1}{|q|}+ x\Big) dt.
\end{equation}
Consider the path spaces
\begin{equation}\label{eq:Omega}
	\Omega=\left\{q\in H^1([0,T],\mathbb{R}^2)\Big| q(0)=(x(0),0), q(T)=(x(T),0), \text{where}
	\  x(0)\ge0,\ x(T)\le 0\right\}.
\end{equation}
Any collision-free minimizer of action functional in $\Omega$ will exhibit a velocity orthogonal to $x$-axis at boundary, which corresponds to a symmetric orbit of Stark problem. 


The proofs for revealing the existence of periodic solutions using variational methods primarily rely on two main procedures. One is the existence of global minimizer, which can be established through the coercivity and weak lower semi-continuous of action functional. This is usually ensured by appropriate choice of the path space. The other one is to exclude possible collisions for the minimizer, which has led to many research on studying collisions of Celestial Mechanics over the past two decades.

The existence of global minimizers in $\Omega$ is relatively straightforward to prove. However, eliminating collisions seems to be a tough problem. Due to the result by Marchal \cite{Marchal2002}, there is no intermediate collision for any minimizer. The challenge lies in excluding potential boundary collisions. Since the Stark problem is a Kepler problem with a linear perturbation, the behavior near collisions aligns with the Kepler problem. So the local perturbation method for Kepler problem should be also applicable for Stark problem. However, our boundary setting in $\Omega$ corresponds an exception case in Kepler problem where the local perturbation does not work. One must search for alternative ways, such as level estimates method.

We find that any attempt to eliminate collisions of minimizer in $\Omega$ are futile. Actually, we prove the following.

\begin{theorem}\label{thm:nonexist}
	Any global minimizer of action functional $\mathcal{A}$ in $\Omega$ corresponds to half of a collision ejection orbit in negative $x$-axis(i.e. the collision ejection orbit has period $2T$). 
\end{theorem}

The idea of proof is the following. As a consequence of Proposition~\ref{pro:Txi>Teta}, there exists no periodic orbit that intersects both the negative and positive $x$-axis once in a period. On the other hand, any minimizer has no boundary collision must correspond to such a periodic orbit. This contradiction implies Theorem~\ref{thm:nonexist}.

We also consider the fixed energy problem. The authors have given some characterization for energy hypersurface with $H<-2$ in \cite{Urs2023}. Although the energy  hypersurface is not compact for $-2\le H<0$, invariant tori exist, and the dynamics on these  invariant tori have some similarity with the compact case.  As perturbed Kepler problem, the averaging method in \cite{Moser1970} ensures the existence of at least $2$ periodic orbits for small $\epsilon$. We show that these two orbits are precisely the two collision orbits on the negative and positive $x$-axis. 

Although the Stark problem is a perturbed Kepler problem, their periodic orbits are dramatically different, especially in the long-term behavior. All periodic orbits in the Kepler problem are elliptic, while these orbits vanish for the Stark problem, except for the two collision orbit in $x$-axis. Due to the existence of small perturbation, the major axis of the elliptic orbits undergoes a deviation over time as shown in Fig.\ref{fig:elliptic}. Interestingly, the system's basin of attraction forms a standard disk, which seems not intuitively obvious. The strange periodic orbits such as brake orbits and $A_k,B_k$s mostly occur in energy level between the critical value $-2\sqrt{\epsilon}$ and $0$. 

In the universe, perturbations from external sources are constantly present. Therefore, the Stark problem seems to be a more realistic model for astronomy.  However, the planar Stark problem is not suitable because orbits usually exhibit very "bad" behavior. Usually, the direction of orbit's rotation around the origin undergoes a reversal like in brake orbit, which seems ridiculous in astronomy. If the orbit is dense in invariant tori, then it will eventually goes as close as possible to the origin at some future time. The orbit is dense in a bounded region between two parabolic curves, whose boundary is not smooth. When approaching  either of two non-smooth point, the aphelion of the orbit will jump from one's neighbourhood to the other. These lead us to consider the spatial Stark problem.

\begin{figure}
	\centering
	\includegraphics[width=0.8\linewidth]{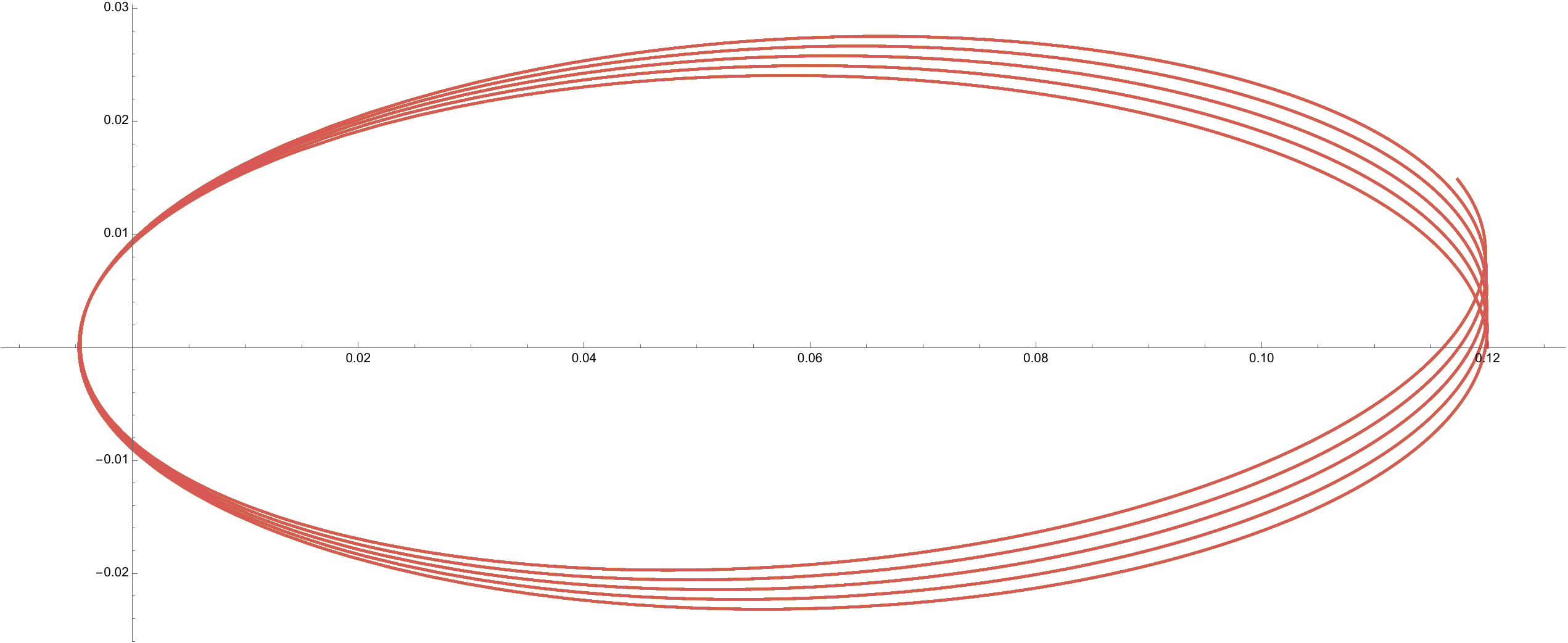}
	\caption{The major axis of the elliptic orbits undergoes a deviation}
	\label{fig:elliptic}
\end{figure}

\subsection{The spatial case}
We consider the equation for Stark problem in $\mathbb{R}^3$  
\begin{equation}\label{eq:stark3}
	\left\{\begin{aligned}
		\ddot{x}&=-\frac{ x}{|q|^3},  \\
		\ddot{y}&=-\frac{ y}{|q|^3}, \\
		\ddot{z}&=-\frac{z}{|q|^3}+1.
	\end{aligned}
	\right.
\end{equation}

The system's invariance under rotation about the $z$-axis implies that the component of angular momentum along the $z$-axis is a conserved quantity, denoted as $L$. When $L=0$, the problem (\ref{eq:stark3}) is actually a planar case. Thus we assume that $L\ne 0$. For spatial case, there will be circular orbits, which is not possible in planar case.
Intuitively, the $z$-axis component of the central attractive force counteracts the external acceleration, while the $xy$-plane component provides the centripetal force for circular motion. 

Fixing $L\ne 0$, the Hamiltonian is still separable by 3D parabolic coordinates
\begin{equation}
	\left\{\begin{aligned}
		x&=\xi\eta\cos \phi,\\
		y&=\xi\eta\sin \phi,\\
		z&=\frac{1}{2}(\xi^2-\eta^2).
	\end{aligned}\right.
\end{equation}

Therefore, the arguments used in the planar case are applicable here as well. The computations become more complicated in the spatial case because we have to deal with cubic polynomial and its  roots. The characteristics of invariant curves depend on parameters such as $L$,$h$ and $c$. With notations introduced in section \ref{sec:spatial}, we give a classification of orbits with different angular momentum $L$ and energy $h$.

\begin{theorem}\label{thm:spatial0}
Given the angular momentum $L$, energy $h$,  constant $c$.  The orbits for the spatial Stark problem (\ref{eq:stark3}) can be classified into the following cases:\\
	(a) When $L^2> \big(\frac{16}{27}\big)^\frac{3}{2}$ or $h>-\frac{3}{2}L^\frac{2}{3}$, all orbits are unbounded.\\
	(b) When $L^2= \big(\frac{16}{27}\big)^\frac{3}{2}$ and $h=-\frac{3}{2}L^\frac{2}{3}$, there is a unique bounded orbit for $c=-\frac{1}{9}$, and it is a circular orbit. \\ 
	(c) When $0<L^2< \big(\frac{16}{27}\big)^\frac{3}{2}$, there exist $h_1,h_2$ with  $h_1<h_2<-\frac{3}{2}L^\frac{2}{3}$ and functions $C_{1\xi}(h)$,  $C_{2,\xi}(h)$,  $C_\eta(h)$ such that the types of bounded orbits can be divided into following subcases depending on both $h$ and $c$. All other orbits are unbounded.
	\begin{enumerate}
	\item [$(c1)$] $h\in (h_2,-\frac{3}{2}L^\frac{2}{3})$,  bounded orbits exist only if $c\in [C_{1\xi}(h),C_{2,\xi}(h)]$, and they are typically oscillatory orbits in space.  $\eta$ is always non-constantly periodic, thus there is no circular orbit. 
	\item [$(c2)$] $h=h_2$, there exist a unstable circular orbit for $c=C_{2,\xi}(h)$ and a orbit that tends to the circular orbit as $t \rightarrow \pm\infty$;
	\item [$(c3)$] $h\in (h_1,h_2)$, the energy surface for $\xi,\eta$ has three connected components, one of which is compact. Just like $h<-2$ in planar case, the compact energy surface is composed of invariant tori given by parameters $c\in [C_{1,\xi}(h),C_\eta(h)]$. These invariant tori degenerate into circles for $c$ at the endpoints of interval;
	\item  [$(c4)$] $h=h_1$, there exists a stable circular orbit  for $c=C_{1,\xi}(h)$.
\end{enumerate}
\end{theorem}
For Theorem~\ref{thm:spatial0}, readers can find a more detailed analysis in Theorem~\ref{thm:spatial}.
The fixed point of $(\xi,\eta)$ corresponds to circular orbits in the initial coordinates. It's evident that $z$ is constant for a circular motion. Depending on the property of $(\xi,\eta)$ as fixed point, we prove the following.
\begin{proposition}[Proposition~\ref{pro:stable}]
	For any $s\in (0,1)$, there exists exactly one circular solution of (\ref{eq:stark3}) with $z=s$. With $L$ being fixed, the circular orbit is stable if and only if $s<\big(\frac{1}{3}\big)^\frac{3}{2}$.
\end{proposition}

The paper is organized as follows.  Section 1 provides an introduction.  Sections 2 and 3 present the proof of main results for the planar Stark problem.  Section 4 addresses the fixed energy problem for the planar case.  Section 5 explores the spatial Stark problem and proves Theorem~\ref{thm:spatial}.   

\section{Invariant curves in separated variables}\label{sec:2}
The so called Arnol’d duality transformation $(x,y,t)\rightarrow (\xi,\eta,\tau)$ was introduced in \cite{LR2011}, which is defined by 
\begin{equation}
	\left\{\begin{aligned}
		\xi^2&=r+x,\\
		\eta^2&=r-x,\\
		dt&=(\xi^2+\eta^2)d\tau=2rd\tau.
	\end{aligned}\right.
\end{equation}
In new coordinates, the Hamiltonian can be expressed as 
\[H=\frac{1}{2}\frac{\xi'^2+\eta'^2}{\xi^2+\eta^2}-\frac{2}{\xi^2+\eta^2}-\frac{\xi^2-\eta^2}{2}.\]
where $\xi'=\frac{d\xi}{d\tau}$ and $\eta'=\frac{d\eta}{d\tau}$.

After multiplying the equation by $\xi^2+\eta^2$ and manipulating the terms, we have
\begin{equation}\label{eq:seperation}
	H\xi^2-\frac{1}{2}\xi'^2+1+\frac{1}{2}\xi^4=-H\eta^2+\frac{1}{2}\eta'^2-1+\frac{1}{2}\eta^4=:-c.
\end{equation}
Note that $H$ is constant for any solution, the first term of (\ref{eq:seperation}) is a function of $\xi$ only, and the second term of (\ref{eq:seperation}) is a function of $\eta$ only,  this implies that each term must be constant.

This first integral corresponds to the conservation of the generalized Laplace-Runge-Lenz vector in the direction of the constant external field. In Cartesian coordinates, it can be expressed as 
\[-c=\dot{y}(x\dot{y}-y\dot{x})-\frac{x}{r}+\frac{1}{2}y^2.\]

Now we have two first integrals and the variables are separated in \eqref{eq:seperation}. Given any $H$ and $c$, the invariant curve of $(\xi,\xi')$ and $(\eta,\eta')$ can be easily determined. By equation \eqref{eq:seperation}
\begin{align}
		&\xi'^2=\xi^4+2H\xi^2+2(c+1),\label{eq:xi'}\\
		&\eta'^2=-\eta^4+2H\eta^2-2(c-1).\label{eq:eta'}
\end{align}
The right hand of quation \eqref{eq:xi'} and \eqref{eq:eta'} are actually quadratic polynomials of $\xi^2$ or $\eta^2$.  Depending on $H$ and $c$, these curves may have different types.

Obviously, these invariant curves are symmetric with respect to coordinate axes. We only need to consider the curve in first quadrant, i.e. the graph defined by 
\begin{equation}
	\xi'=\sqrt{\xi^4+2H\xi^2+2(c+1)}.
\end{equation}
Denote
\[\Delta_\xi=(2H)^2-8(c+1).\]
In Fig.\ref{fig:xi}, we list all possible types of invariant curves in $(\xi,\xi')$ coordinates.
\begin{figure}
	\begin{subfigure}{0.3\textwidth}
		\includegraphics[width=\linewidth]{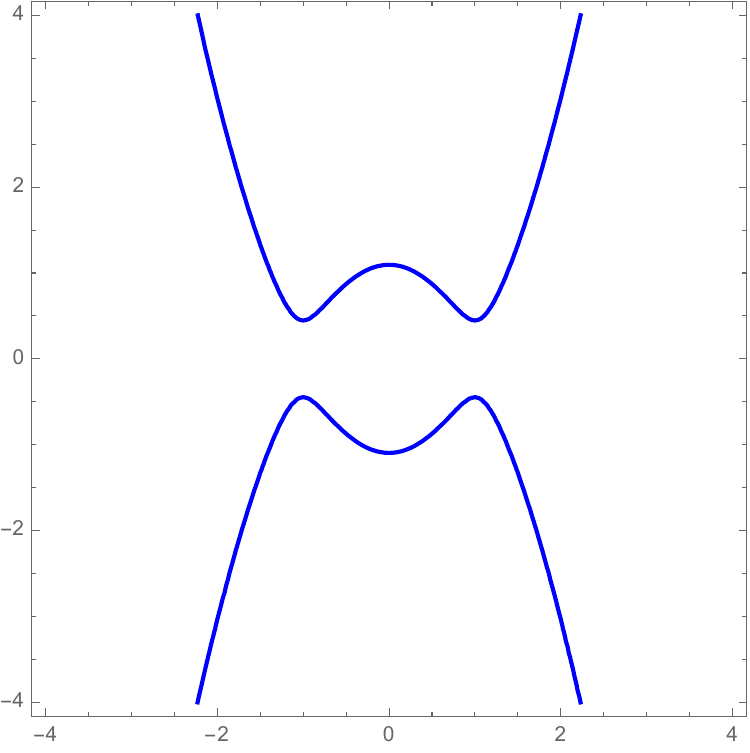}
		\caption{$H<0, \Delta_\xi<0$}
		\label{fig:xi_a}
	\end{subfigure}
    \hfill
	\begin{subfigure}{0.3\textwidth}
		\includegraphics[width=\linewidth]{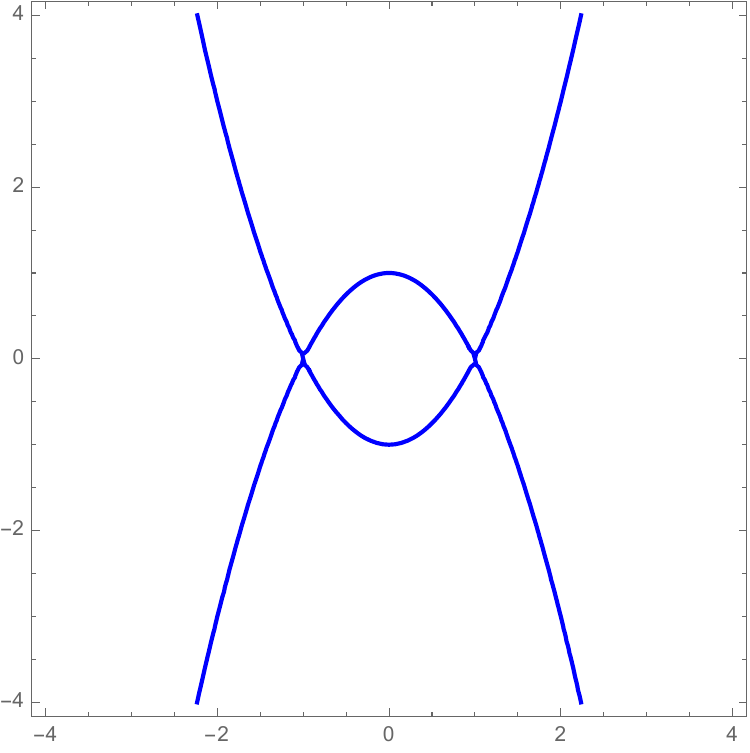}
		\caption{$H<0, \Delta_\xi=0$}
		\label{fig:xi_b}
	\end{subfigure}
	\hfill
	\begin{subfigure}{0.3\textwidth}
		\includegraphics[width=\linewidth]{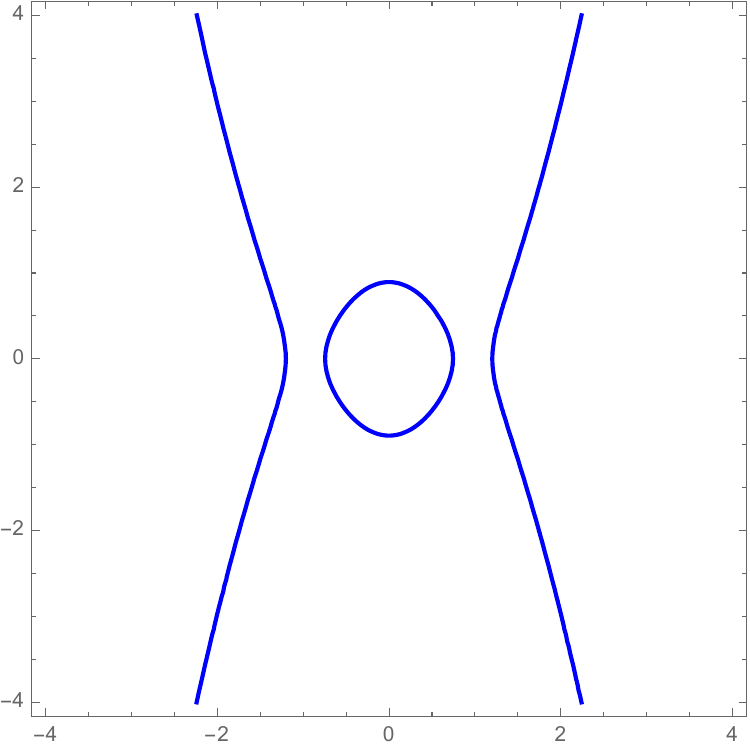}
		\caption{$H<0, \Delta_\xi>0, c>-1$}
		\label{fig:xi_c}
	\end{subfigure}
	\hfill
	\begin{subfigure}{0.3\textwidth}
		\includegraphics[width=\linewidth]{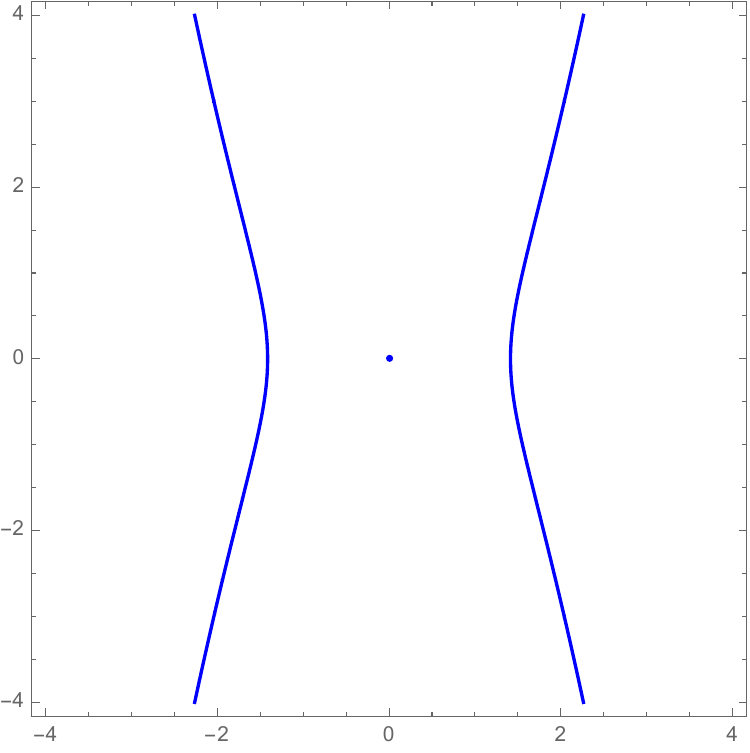}
		\caption{$H<0, c=-1$}
		\label{fig:xi_d}
	\end{subfigure}
    \hfill
    \begin{subfigure}{0.3\textwidth}
    	\includegraphics[width=\linewidth]{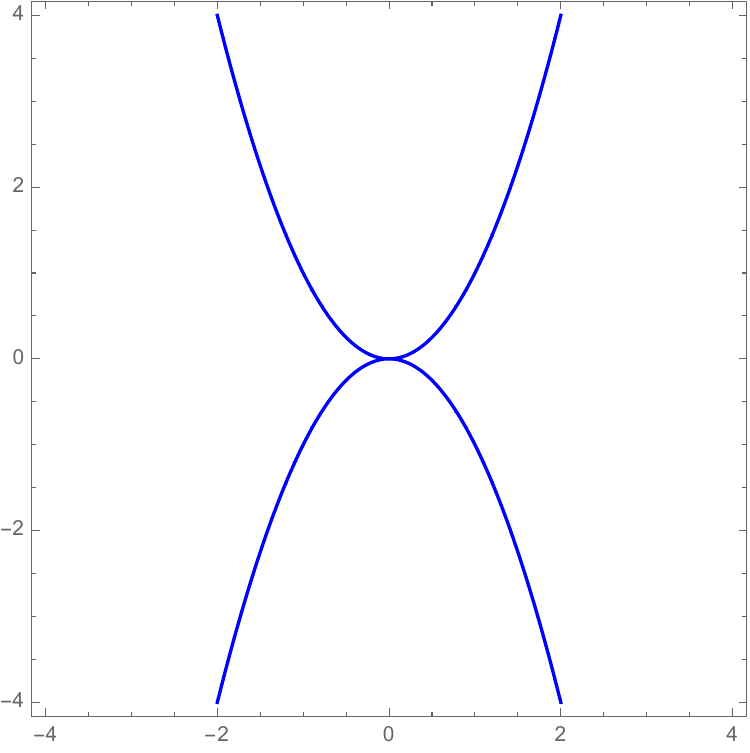}
    	\caption{$H=0, c=-1$}
    	\label{fig:xi_e}
    \end{subfigure}
    \hfill
	\begin{subfigure}{0.3\textwidth}
	\includegraphics[width=\linewidth]{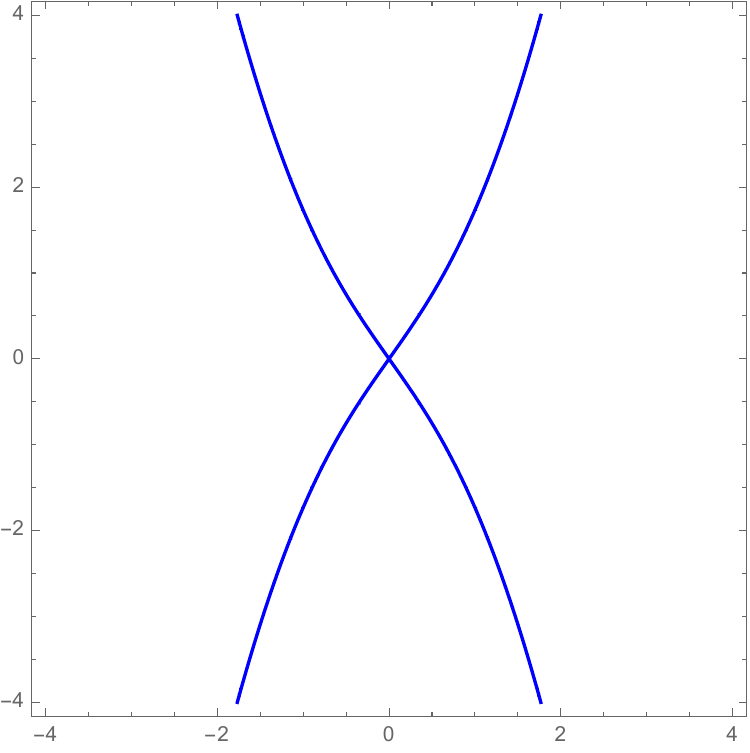}
	\caption{$H>0, c=-1$}
	\label{fig:xi_f}
    \end{subfigure}
    \hfill
    \begin{subfigure}{0.3\textwidth}
	\includegraphics[width=\linewidth]{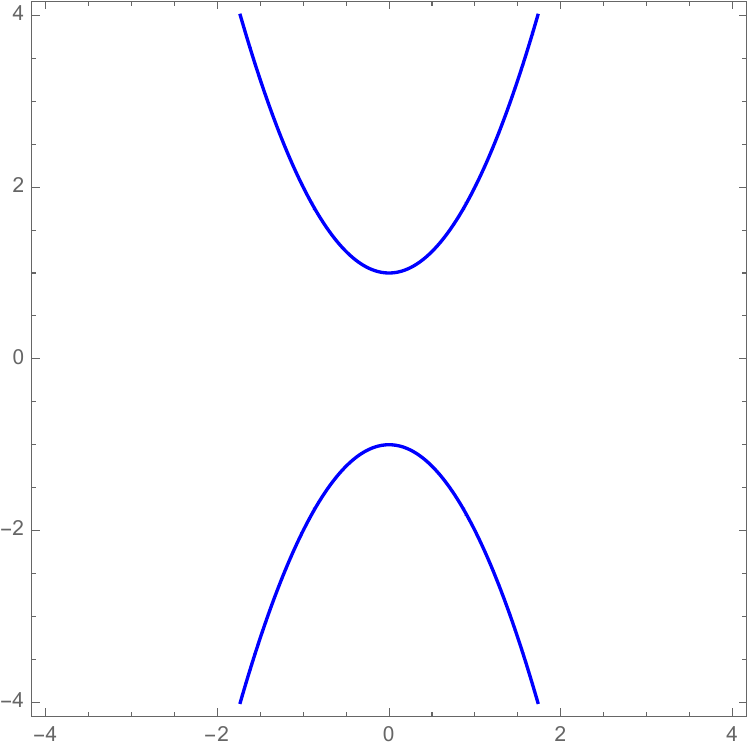}
	\caption{$H\ge0, c>-1$}
	\label{fig:xi_g}
    \end{subfigure}
    \quad\ \ 
    \begin{subfigure}{0.3\textwidth}
	\includegraphics[width=\linewidth]{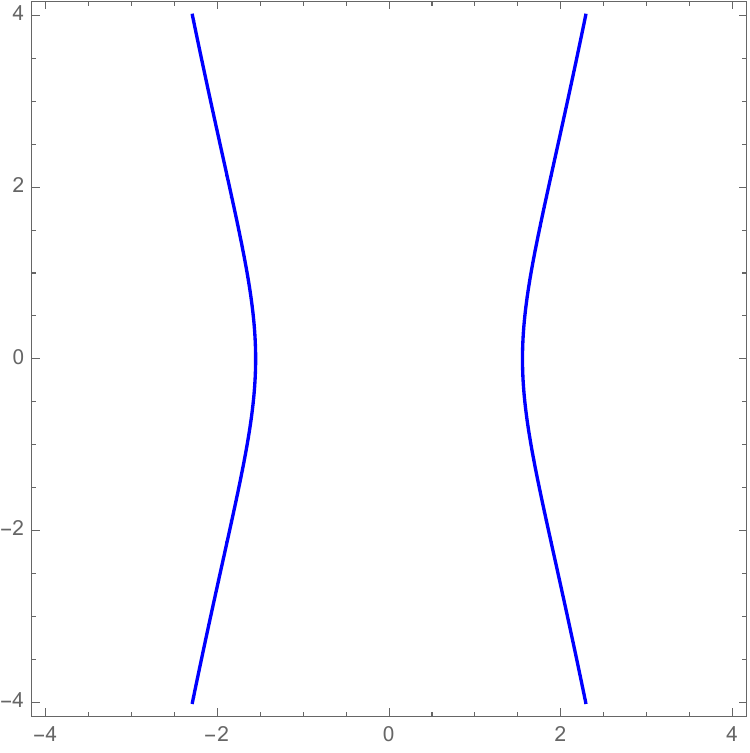}
	\caption{$c<-1$}
	\label{fig:xi_h}
    \end{subfigure}
   	
	\caption{Different types of invariant curves in $(\xi,\xi')$ coordinates.}
	\label{fig:xi}
\end{figure}

Similarly, denote $\Delta_\eta=(2H)^2-8(c-1)=\Delta_\xi+16$. Then all possible types of invariant curves in $(\eta,\eta')$ coordinate are illustrated in Fig.\ref{fig:eta}. 

\begin{figure}
	\begin{subfigure}{0.3\textwidth}
		\includegraphics[width=\linewidth]{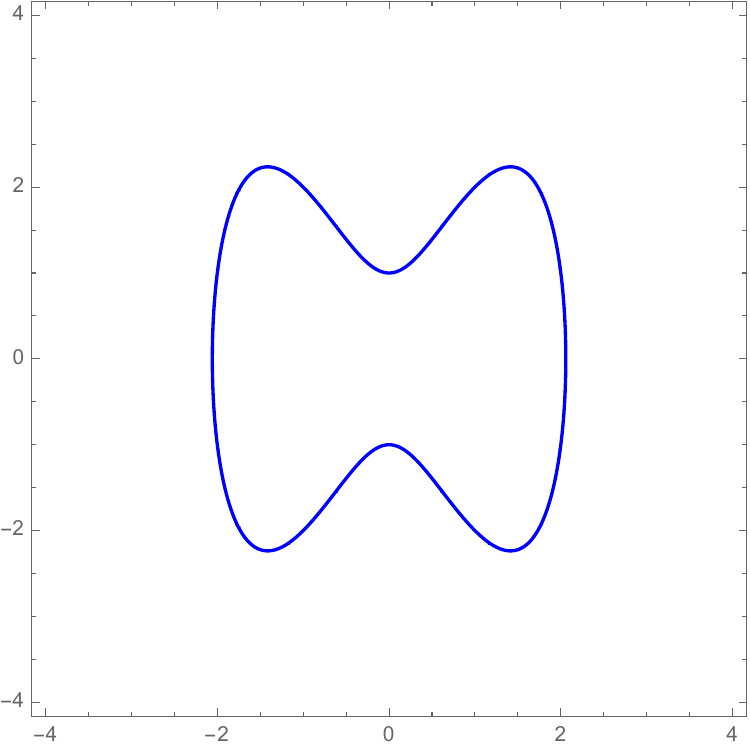}
		\caption{$H>0, c<1$}
		\label{fig:eta_a}
	\end{subfigure}
	\hfill
	\begin{subfigure}{0.3\textwidth}
		\includegraphics[width=\linewidth]{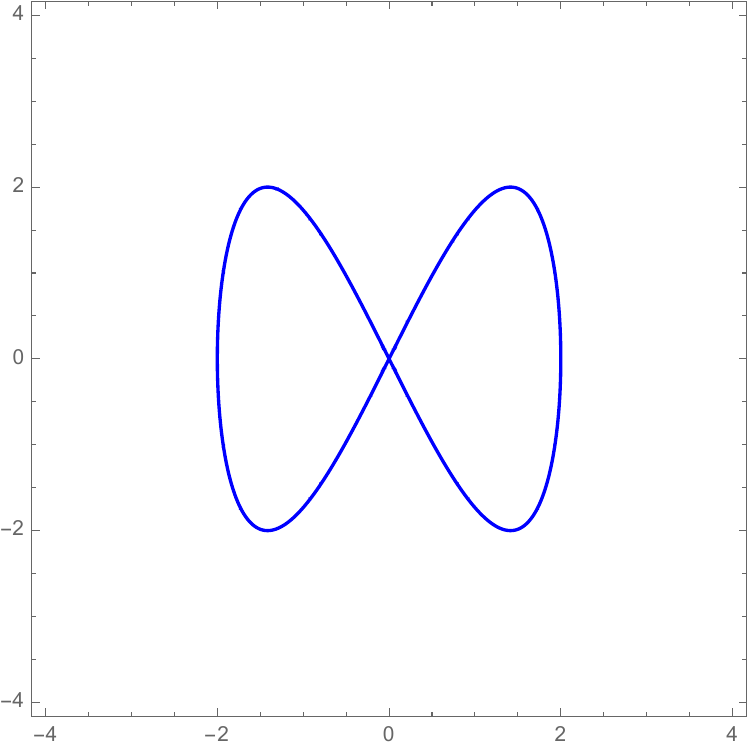}
		\caption{$H>0, c=1$}
		\label{fig:eta_b}
	\end{subfigure}
	\hfill
	\begin{subfigure}{0.3\textwidth}
		\includegraphics[width=\linewidth]{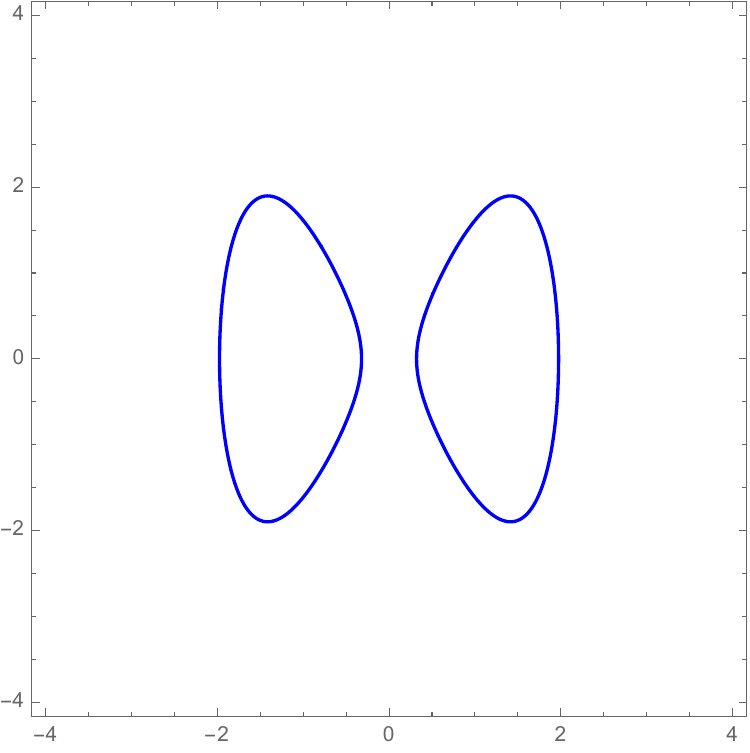}
		\caption{$H>0, \Delta_\xi>0, c>1 $}
		\label{fig:eta_c}
	\end{subfigure}
	\hfill
\begin{subfigure}{0.3\textwidth}
	\includegraphics[width=\linewidth]{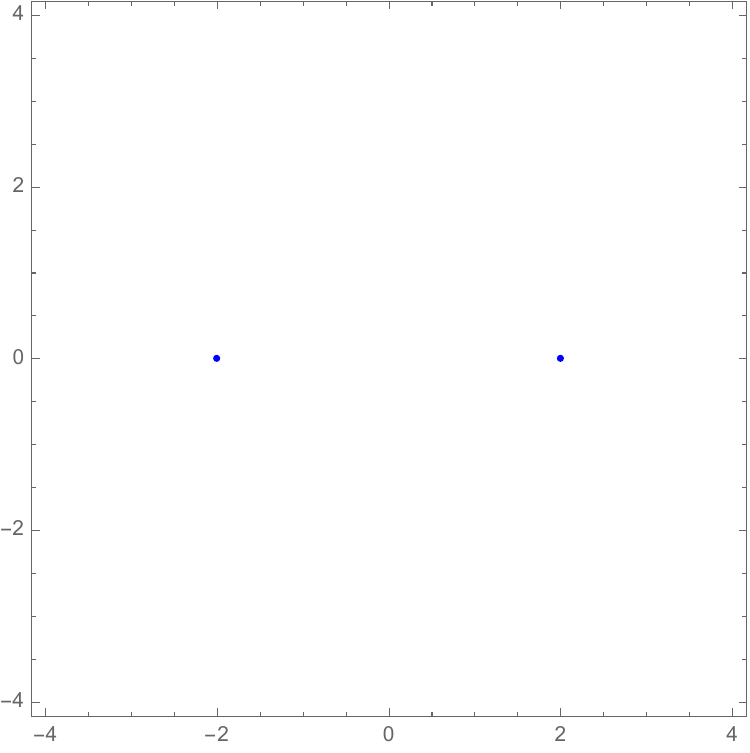}
	\caption{$H>0,  \Delta_\xi=0$}
	\label{fig:eta_d}
\end{subfigure}
	\hfill
	\begin{subfigure}{0.3\textwidth}
		\includegraphics[width=\linewidth]{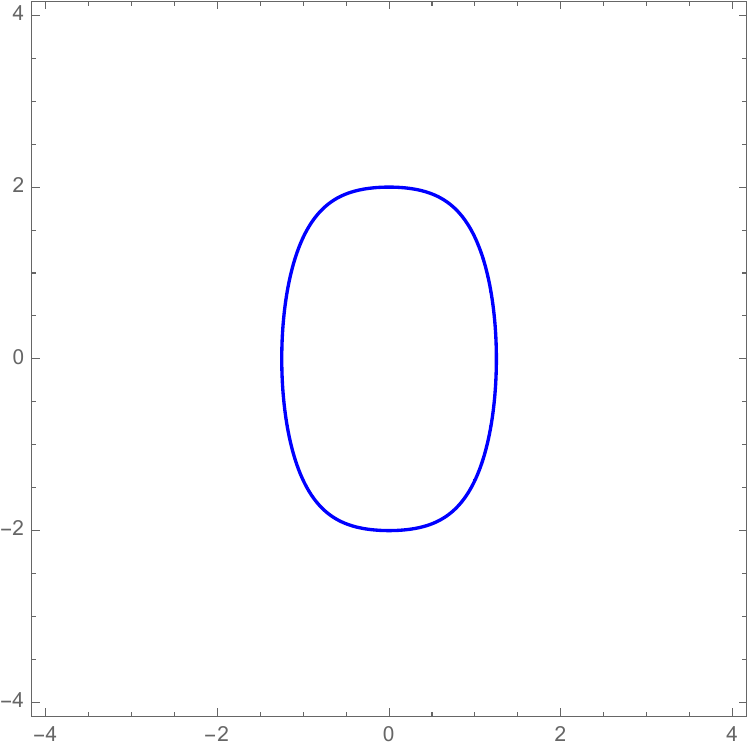}
		\caption{$H\le0, c<1$}
		\label{fig:eta_e}
	\end{subfigure}
	\hfill
	\begin{subfigure}{0.3\textwidth}
		\includegraphics[width=\linewidth]{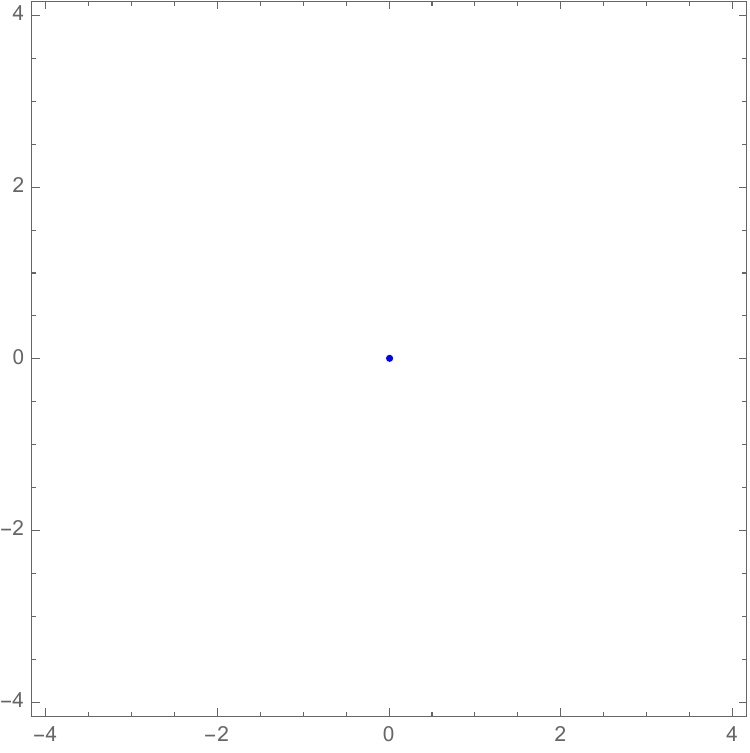}
		\caption{$H\le0, c=1$}
		\label{fig:eta_f}
	\end{subfigure}

	\caption{Different types of invariant curves in $(\eta,\eta')$ coordinates.}
	\label{fig:eta}
\end{figure}

Note that $H$ and $c$ are conserved throughout the motion. For any given initial condition, $H$ and $c$ are determined. One can readily identify the invariant curves for $(\xi,\xi')$ and $(\eta,\eta')$. The periodic solutions of the Stark problem are precisely those periodic solutions in $(\xi,\eta,\tau)$ coordinates. Thus, the solution is periodic if and only if both $\xi$ and $\eta$ are periodic, and their periods are rationally dependent.

\begin{remark}
	It should be noted that the mapping $(\xi,\eta) \rightarrow (x,y)$ is a four-to-two relationship whenever both $\xi$ and $\eta$ are non-zero. In general, each solution in $(\xi,\eta)$ coordinates corresponds to two solutions in $(x,y)$ coordinates, symmetric with respect to the $x$-axis. Due to the system's natural symmetry, we consider the two symmetric solutions in $(x,y)$ coordinates as a single entity. Consequently, for any orbit in $(\xi,\eta)$ coordinates, we obtain a unique orbit in $(x,y)$ coordinates.
\end{remark}

In fact, G. Lantoine and R. Russell in \cite{LR2011} derived explicit formulas for $\xi$ and $\eta$ from \eqref{eq:xi'} and \eqref{eq:eta'}. These formulas involve Jacobi elliptic functions. Here what we need are the period of $\xi$ and $\eta$, which are in forms of the elliptic integral of the first kind.

Recall that the complete elliptic integral of the first kind $K$ is given by 
\[K(m)=\int_{0}^{1}\frac{dt}{\sqrt{(1-t^2)(1-mt^2)}}.\]
Consider the first quadrant of $(\xi,\xi')$ and $(\eta,\eta')$. By \eqref{eq:xi'} and \eqref{eq:eta'}
\begin{align}
		\frac{d\tau}{d \xi}&=\frac{1}{\sqrt{\xi^4+2H\xi^2+2(c+1)}}, \label{eq:tauxi}\\
		\frac{d\tau}{d \eta}&=\frac{1}{\sqrt{-\eta^4+2H\eta^2-2(c-1)}}. \label{eq:taueta}
\end{align}
Suppose $\xi$ and $\eta$ are periodic,  by integrating \eqref{eq:tauxi} and \eqref{eq:taueta}, we can express their periods $T_\xi$ and $T_\eta$ in terms of elliptic integrals. The following proposition is part of the results in \cite{LR2011}. 

\begin{proposition}\label{pro:T}
	If the invariant curve $(\xi,\xi')$ satisfying (\ref{eq:xi'}) is the periodic case in Fig.\ref{fig:xi_c}, then $T_\xi=\frac{4K(\frac{\xi_1^2}{\xi_2^2})}{{\xi_2}}$, where $\xi_1^2=-H-\sqrt{H^2-2(c+1)}$ and $\xi_2^2={-H+\sqrt{H^2-2(c+1)}}$. 
	If the invariant curve $(\eta,\eta')$ satisfying (\ref{eq:eta'}) is the periodic case in Fig.\ref{fig:eta_e}, then $T_\eta=\frac{4K(\frac{\eta_1^2}{\eta_1^2+\eta_2^2})}{\sqrt{\eta_1^2+\eta_2^2}}$, where $\eta_1^2=H+\sqrt{H^2-2(c-1)}$ and $-\eta_2^2=H-\sqrt{H^2-2(c-1)}$.
\end{proposition}
\begin{proof}
	When Fig.\ref{fig:xi_c} happens,  we have
	\[\xi^4+2H\xi^2+2(c+1)=(\xi^2-\xi_1^2)(\xi^2-\xi_2^2)\]
	where $0<\xi_1<\xi_2$ and $\xi_1^2=-H-\sqrt{H^2-2(c+1)}$ and $\xi_2^2={-H+\sqrt{H^2-2(c+1)}}$ are two solutions of $s^2+2Hs+2(c+1)=0$. 
	
	By \eqref{eq:tauxi},
	\begin{align*}
		\frac{T_\xi}{4}&=\int_{0}^{\xi_1}\frac{d\xi}{\sqrt{\xi^4+2H\xi^2+2(c+1)}},\\
		&=\int_{0}^{\xi_1}\frac{d\xi}{\sqrt{(\xi^2-\xi_1^2)(\xi^2-\xi_2^2)}},\\
		&=\int_{0}^{1}\frac{dz}{\sqrt{\xi_2^2(1-z^2)(1-\frac{\xi_1^2}{\xi_2^2}z^2)}},\\
		&=\frac{K(\frac{\xi_1^2}{\xi_2^2})}{{\xi_2}}.
	\end{align*}
	where $K(m)=\int_{0}^{1}\frac{dt}{\sqrt{(1-t^2)(1-mt^2)}}$ is elliptic integral of the first kind. 

Hence,
 \begin{equation}\label{eq:Txi}
	T_\xi=\frac{4K(\frac{\xi_1^2}{\xi_2^2})}{{\xi_2}}.
\end{equation}
As proved in \cite{LR2011}, $\xi$ can be expressed as the inverse of the elliptic integral, which is the Jacobi elliptic function
\[\xi=\xi_1\text{sn}(\xi_2\tau,\frac{\xi_1^2}{\xi_2^2}).\]

By similar computation, when Fig.\ref{fig:eta_e} happens, we have
\begin{equation}\label{eq:Teta}
	T_\eta=\frac{4K(-\frac{\eta_1^2}{\eta_2^2})}{\eta_2}=\frac{4K(\frac{\eta_1^2}{\eta_1^2+\eta_2^2})}{\sqrt{\eta_1^2+\eta_2^2}},
\end{equation}
where $0<\eta_1\le\eta_2$ and $\eta_1^2=H+\sqrt{H^2-2(c-1)}$ and $-\eta_2^2=H-\sqrt{H^2-2(c-1)}$ are two solutions of $-s^2+2Hs-2(c-1)=0$. $\eta$ can be expressed as 
\[\eta=\eta_1\text{cn}(\sqrt{\eta_1^2+\eta_2^2}\tau,\frac{\eta_1^2}{\eta_1^2+\eta_2^2}).\]
\end{proof}

\section{Proof of main results for the planar case}
In this section, we are prepared to prove the existence of various symmetric periodic orbit. In order to get a symmetric orbit, it's natural to initiate the orbit perpendicularly from the $x$-axis, which is the starting point of our results. Next, we will give the proofs for Theorem~\ref{thm:brake} - Theorem~\ref{thm:nonexist}.

\textbf{Proof of Theorem~\ref{thm:brake}:} Let $q=(x,y)$  be the solution of (\ref{eq:stark}) with initial condition $(x(0),y(0))=(a,0)$ and $(\dot{x}(0),\dot{y}(0))=(0,b)$.  Then $H$ and $c$ can be expressed as 
\begin{equation}\label{eq:Hc}
	\begin{aligned}
		H=\frac{1}{2}b^2-\frac{1}{a}-a \ \quad \text{ and } \ \quad 
	    c=1-ab^2.
	\end{aligned}
\end{equation}

After classifying the invariant curves for both $\xi$ and $\eta$, we observe that all possible invariant curves for $(\eta, \eta')$ are bounded, and they exhibit periodic behavior, except for the case illustrated in Fig.\ref{fig:eta_b}. To find a periodic solution, we must focus on the cases when $(\xi,\xi')$ are periodic.
In our setting, there holds
\begin{equation}
	\Delta_\xi=(2H)^2-8(c+1)=(b^2-\frac{2}{a}+2a)^2\ge 0.
\end{equation}
The possible periodic invariant curves for $(\xi,\xi')$ are depicted in (b)-(f) of Fig.\ref{fig:xi}. 

 When (d)-(f) of Fig.\ref{fig:xi} happen, we observe that $\eta$ is always periodic,  and $(\xi,\xi')$ is periodic if and only if $(\xi,\xi')\equiv (0,0)$.  Recall that $\xi=r+x$, $\xi=0$ implies that $q$ is always on the negative $x$-axis. Thus $q$ corresponds to a collision orbit in negative $x$-axis.
 
 When the case in Fig.\ref{fig:xi_b} happens, that is $a\in (0,1]$ and $b=\sqrt{\frac{2}{a}-2a}$,  we have $c=2a^2-1$ and $H=-2a$. The possible invariant curves for $(\xi,\xi')$ are two hyperbolic fixed points, two heteroclinic orbits and four unbounded invariant curves. $(\xi,\xi')$ is periodic only when it is a fixed point. 
 
 If $a=1$, then $(\eta,\eta')\equiv (0,0)$ as the case in Fig.\ref{fig:eta_f}. The orbit corresponds to the unique equilibrium $(1,0)$. 
 
 If $a\in (0,1)$, then the invariant curve for $(\eta,\eta')$ is as the case in Fig.\ref{fig:eta_e}. This implies $q$ is a non-collision periodic solution. Note that
 \begin{equation}\label{eq:velocity0}
 	\dot{x}=0,\dot{y}=0 \Longleftrightarrow\xi'=0, \eta'=0. 
 \end{equation}
Clearly, there exist a instance such that $\xi'=0$ and $\eta'=0$. Hence there exist some time when $\dot{q}=(\dot{x},\dot{y})=(0,0)$. The periodic orbit $q$ is actually a brake orbit.

Substituting the values of $H$ and $c$ into the equation \eqref{eq:eta'} and let $\eta'=0$, we have
\[\eta^2=2-2a.\]
Recall that $\xi^2\equiv2a$ by the initial condition. When $\eta'=0$,
\begin{equation}
	2r=\xi^2+\eta^2=2.
\end{equation}
Consequently, when $\dot{q}=0$, there holds $r=1$.  The properties $(ii)$-$(vi)$ of Theorem~\ref{thm:brake} is proved.

When the case in Fig.\ref{fig:xi_c} happens, $\Delta_\xi>0$. By \eqref{eq:xi'}, the intersection points of invariant curves and $\xi$-axis satisfies 
\[\xi^2=\frac{-2H\pm\sqrt{\Delta_\xi}}{2}.\]
In our setting, the  solutions for $\xi'=0$ are $\xi_1^2=2a$ and $\xi_2^2=\frac{2}{a}-b^2$.

If $b<\sqrt{\frac{2}{a}-2a}$, then $2a<\frac{2}{a}-b^2$, our initial condition $(2a,0)$ for $(\xi,\xi')$ is on the bounded curve of Fig.\ref{fig:xi_c}. If $b>\sqrt{\frac{2}{a}-2a}$, then $(2a,0)$ is on the unbounded curve of Fig.\ref{fig:xi_c}. Since $q$ is bounded if and only if $\xi^2+\eta^2=2r$ is bounded, $q$ is bounded if and only if $b\le\sqrt{\frac{2}{a}-2a}$. The property $(i)$ is proved.

Note that $(\xi,\eta)\rightarrow (x,y)$ is a four to two map, the two hyperbolic fixed point for $(\xi,\xi')$ corresponds two the same brake orbits. The two heteroclinics orbit connecting two fixed points corresponds the same orbit in $(x,y)$ coordinates, which is a orbit which tends to brake orbit as $t\rightarrow \pm\infty$. By different choice of initial phases for $\xi$ and $\eta$, we actually have a family of such orbits. The property $(v)$ is proved.

By definition $\xi^2=r+x$, $q$ is in negative $x$-axis if $\xi=0$. Note that for bounded orbits, it necessary that $\xi$ is bounded. By classification of invariant curves for $(\xi,\xi')$ in Fig.\ref{fig:xi}, the hyperbolic fixed points in Fig.\ref{fig:xi_b} are the only invariant curves which are bounded and do not intersect $\xi=0$. These are exactly  the brake orbits in $(ii)$. The proof of Theorem~\ref{thm:brake} is complete.
\qed

\textbf{Proof of Corollary~\ref{cor:1}:} Note that all the bounded solution for Stark problem are considered in the proof of Theorem~\ref{thm:brake}. If $c=-1$, then the solution is bounded if and only if it is a collision orbit in negative $x$-axis. For other bounded solutions, the invariant curves for $\xi$ must be in Fig.\ref{fig:xi_b} or Fig.\ref{fig:xi_c}. Note that $2r=\xi^2+\eta^2<\xi_1^2+\eta_1^2$, it sufficient to prove that 
\begin{equation}\label{eq:xi2eta2}
	\xi_1^2+\eta_1^2\le 2.
\end{equation}
Note that $\xi_1^2={-H-\sqrt{H^2-2(c+1)}}$, $\eta_1^2=H+\sqrt{H^2-2(c-1)}$. \eqref{eq:xi2eta2} is equivalent to 
\[\sqrt{H^2-2(c-1)}-\sqrt{H^2-2(c+1)}\le 2,\]
which is obviously correct, since the difference between two terms under square root is 4. 

The equality holds if and only if $H^2=2(c+1)$, which is satisfied for these brake orbit in Theorem~\ref{thm:brake}.
\qed

By proof of Theorem~\ref{thm:brake}, given $a\in (0,1)$ and $b\in(0,\sqrt{\frac{2}{a}-2a})$, the cases in Fig.\ref{fig:xi_c} and Fig.\ref{fig:eta_e} happen. By Proposition~\ref{pro:T}
\[T_\xi=\frac{4K(\frac{\xi_1^2}{\xi_2^2})}{{\xi_2}},\ \ \ \  T_\eta=\frac{4K(-\frac{\eta_1^2}{\eta_2^2})}{\eta_2^2}.\]

For simplicity, we introduce the following function as the author did in \cite{Urs2023}
\begin{equation*}
	\Phi:(-\infty,1)\rightarrow \mathbb{R},\ \ \ \ x\rightarrow \frac{1}{\sqrt{1+\sqrt{1-x}}}K\Big(\frac{1-\sqrt{1-x}}{1+\sqrt{1-x}}\Big).
\end{equation*}
Then 
\begin{equation}\label{eq:TinPhi}
	T_\xi=\frac{4}{\sqrt{|H|}}\Phi\big(\frac{2c+2}{H^2}\big),\quad T_\eta=\frac{4}{\sqrt{|H|}}\Phi\big(\frac{2c-2}{H^2}\big).
\end{equation}
Direct computation shows that
\begin{equation}
	\begin{aligned}
		\Phi'(x)=\frac{1}{4\sqrt{(1-x)(1+\sqrt{1-x})^3}}K\Big(\frac{1-\sqrt{1-x}}{1+\sqrt{1-x}}\Big) \\
		+\frac{1}{\sqrt{(1-x)(1+\sqrt{1-x})^5}}K'\Big(\frac{1-\sqrt{1-x}}{1+\sqrt{1-x}}\Big)
	\end{aligned}	
\end{equation}
Thus $\Phi(x)$ is a strictly increasing function of $x$. Together with \eqref{eq:TinPhi}, we have proved
\begin{proposition}\label{pro:Txi>Teta}
	Whenever both $\xi$ and $\eta$ are periodic, it always holds $T_\xi>T_\eta$.
\end{proposition}

In order to prove Theorem~\ref{thm:Ak} and Theorem~\ref{thm:Bk}, we need the following two propositions.
\begin{proposition}\label{pro:K}
	In the setting of Theorem~\ref{thm:Ak}, fixing any $a\in (0,1)$ and let $(\dot{x}(0),\dot{y}(0))=(0,b)$, then $T_\xi$ and $T_\eta$ are functions of $b$ in $[0,\sqrt{\frac{2}{a}-2a})$. $\frac{T_\xi}{T_\eta}$ a strictly increasing function of $b$ and satisfies
	\begin{equation}\label{eq:T_xi/T_eta}
		\alpha(a):=\lim_{b\rightarrow 0}\frac{T_\xi}{T_\eta}=\frac{\sqrt{1+a^2}K(a^2)}{K(0)},\ \ \ \ \lim_{b\rightarrow \sqrt{\frac{2}{a}-2a}}\frac{T_\xi}{T_\eta}=+\infty.
	\end{equation}
\end{proposition}
\begin{proof}
	
By direct computation, one find that
\begin{equation}\label{eq:bto0}
	\lim_{b\rightarrow 0}T_\xi=\frac{4K(a^2)}{\sqrt{\frac{2}{a}}},\ \ \ \ \lim_{b\rightarrow 0}T_\eta=\frac{4K(0)}{\sqrt{\frac{2}{a}+2a}}
\end{equation}
and 
 \begin{equation}\label{eq:bto+}
 	\lim_{b\rightarrow \sqrt{\frac{2}{a}-2a}}T_\xi=+\infty,\quad\lim_{b\rightarrow \sqrt{\frac{2}{a}-2a}}T_\eta=2K(\frac{1-a}{2}).
 \end{equation}
The equation \eqref{eq:T_xi/T_eta} follows from \eqref{eq:bto0} and \eqref{eq:bto+}. It suffices to proof the monotonicity of $\frac{T_\xi}{T_\eta}$.

By \eqref{eq:TinPhi},
\[\frac{T_\xi}{T_\eta}=\frac{\Phi\big(\frac{2c+2}{H^2}\big)}{\Phi\big(\frac{2c-2}{H^2}\big)},\]
where both $H$ and $c$ are functions of $b$ given in \eqref{eq:Hc}.

Note that $H<0$ and $0\le b<\sqrt{\frac{2}{a}-2a}$, we have
\begin{equation*}
	\begin{aligned}
		\Big(\frac{2c+2}{H^2}\Big)'=\frac{2c'H-4H'(c+1)}{H^3}
		=\frac{2ab}{H^3}(2a-\frac{2}{a}+b^2)>0.
	\end{aligned}
\end{equation*}
and 
\begin{equation*}
	\begin{aligned}
		\Big(\frac{2c-2}{H^2}\Big)'=\frac{2c'H-4H'(c-1)}{H^3}
		=\frac{2ab}{H^3}(2a+\frac{2}{a}+b^2)<0.
	\end{aligned}
\end{equation*}
Since $\Phi(x)$ is a strictly increasing function of $x$, $\frac{T_\xi}{T_\eta}$ is strictly increasing function of $b$.
\end{proof}

\begin{proposition}\label{pro:K1}
	In the setting of Theorem~\ref{thm:Bk}, fixing any $a'\in (-1,1)$ and let $(x(0),y(0))=(a',d)$, then $T_\xi$ and $T_\eta$ are functions of $d\in(0,\sqrt{1-a'^2})$. $\frac{T_\xi}{T_\eta}$ a monotone function of $d$.
\end{proposition}
\begin{proof}
	The proof is similar to the proof of Proposition~\ref{pro:K}. Now $H$ and $c$ is function of $d$ given by
	\begin{equation}\label{eq:Hcind}
		\begin{aligned}
			H=-\frac{1}{\sqrt{a'^2+d^2}}-a' \ \quad \text{ and } \quad 
			c=\frac{a'}{\sqrt{a'^2+d^2}}-\frac{d^2}{2}.
		\end{aligned}
	\end{equation}

We have
\begin{equation*}
	\begin{aligned}
		\Big(\frac{2c+2}{H^2}\Big)'&=\frac{2c'H-4H'(c+1)}{H^3}
		=\frac{2d}{H^3}\big(1-\frac{1}{r^2}\big)\big(\frac{2}{r}+a'+\frac{a'}{r^2}\big)>0.
	\end{aligned}
\end{equation*}
and 
\begin{equation*}
	\begin{aligned}
		\Big(\frac{2c-2}{H^2}\Big)'&=\frac{2c'H-4H'(c-1)}{H^3}
		=\frac{2d}{H^3}\big(1+\frac{1}{r^2}\big)\big(\frac{2}{r}+a'-\frac{a'}{r^2}\big)<0.
	\end{aligned}
\end{equation*}
where $r=\sqrt{a'^2+d^2}\in (a',1)$. 

Thus, $\frac{T_\xi}{T_\eta}$ is a strictly increasing function of $d\in (0,\sqrt{1-a'^2})$.
\end{proof}

Note that $q$ lies in $x$-axis if and only if $\xi=0$ or $\eta=0$. If the orbit of $q$ is perpendicular to $x$-axis, then either $\xi'=0$ or $\eta'=0$ at the point of intersection. In our setting, both $\xi$ and $\eta$ are periodic but not constant. Hence, $\xi$ and $\xi'$ cannot be zero simultaneously, and the same applies to $\eta$ and $\eta'$. In other words, the orbit of $q$ is perpendicular to $x$-axis if and only if one of the vectors $(\xi,\xi'),(\eta,\eta')$ lies on the horizontal axis, and the other one on the vertical axis. 
Our initial condition is perpendicular to $x$-axis and the corresponding point in $(\xi,\xi')$ and $(\eta,\eta')$ coordinates satisfies 
\[(\xi^2,\xi'^2)=(2a,0),\ \  \ \ (\eta^2,\eta'^2)=(0,2ab^2).\] 

If one finds another time such that the orbit of $q$ is perpendicular to $x$-axis, then $q$ is a  symmetric periodic orbit. It suffices to find some $\tau>0$ such that one of the vectors $(\xi,\xi'),(\eta,\eta')$ lies on the horizontal axis, and the other one on the vertical axis. 

Similarly, if there exists some $\tau>0$ such that $\xi=0$ and $\eta=0$, then $q$ is a symmetric collision orbit. By \eqref{eq:velocity0}, if there exists some $\tau>0$ such that $\xi'=0$ and $\eta'=0$, then $q$ is a symmetric brake orbit.

Proof of Theorem~\ref{thm:Ak} and Theorem~\ref{thm:Bk} are similar. Here we only give a proof for Theorem~\ref{thm:Ak}. Proposition~\ref{pro:K} and Proposition~\ref{pro:K1} are needed for the proofs of Theorem~\ref{thm:Ak} and Theorem~\ref{thm:Bk}, respectively.

\textbf{Proof of Theorem~\ref{thm:Ak}:} Note that $\xi=0$ if and only if $q$ is in negative $x$-axis, while $\eta=0$ if and only if $q$ is in positive $x$-axis. In order to find the periodic orbit $A_k$, one need the orbit for $(\eta,\eta')$ intersects $\eta=0$ $k$ times before $\xi$ reaches $0$. And when $\xi=0$ for the first time, $\eta'=0$(i.e. $q$ is perpendicular to negative $x$-axis). This is equivalent to the following equation.
\begin{equation}\label{eq:Ak}
	\frac{T_\xi}{4}=\frac{kT_\eta}{2}+\frac{T_\eta}{4}\Longleftrightarrow T_\xi=(2k+1)T_\eta.
\end{equation}

By Proposition~\ref{pro:K}, $\frac{T_\xi}{T_\eta}$ can take all values within the interval $(\alpha(a),+\infty)$. Let $k_p$ be the smallest integer such that $2k_p+1\in (\alpha(a),+\infty)$. Then for any $k\ge k_p$, one can find $b_k$ such that  \eqref{eq:Ak} holds. $A_k$ is then the corresponding solution.

Similarly, finding the periodic orbit $B_k$ is equivalent to finding $b_k'$ such that 
\begin{equation}\label{eq:Bk}
	T_\xi=(2k+2)T_\eta.
\end{equation}
For above $k_p$, such point $b_k'$ certainly exists. 

Since $T_\eta$ is bounded from below for any $a\in (0,1)$, as $k\rightarrow +\infty$, $T_\xi\rightarrow +\infty$. Thus both $b_k$ and $b_k'$ tends to $\sqrt{\frac{2}{a}-2a}$. It's easy to see that $\{b_k\}$ and $\{b_k'\}$ satisfies the inequality
\[\dots<b_k<b_k'<b_{k+1}<b_{k+1}'<\dots.\]
By \eqref{eq:T_xi/T_eta},
\[\lim_{a\rightarrow 1}\alpha(a)=+\infty.\]
Thus, $$\lim_{a\rightarrow 1}k_p=+\infty$$.
\qed

\textbf{Proof of Theorem~\ref{thm:nonexist}:} Now we consider the following minimizing problem 
\begin{equation}\label{eq:infA}
	\inf_{q\in \Omega}\mathcal{A}(q)
\end{equation}
where $\mathcal{A}$ and $\Omega$ are defined in \eqref{eq:lagrange} and \eqref{eq:Omega}. According to the results in \cite{Marchal2002}, there are no intermediate collisions. Our proof consists of two steps: The first step is to prove that any minimizer has collision at boundary. The second step is to prove that the minimizer must be half of a collision ejection orbit in negative $x$-axis.

\textbf{Step 1:} By standard argument, one finds the action function to be coercive in $\Omega$, i.e. $A(q)\rightarrow +\infty$ as $||q||\rightarrow +\infty$. $\mathcal{A}$ is weakly lower semi-continuous, and $\Omega$ is weakly closed. Thus $\mathcal{A}$ is attains its minimum on $\Omega$. Assume $q^*$ is a collision-free action minimizer. By the "first variation orthogonality", the velocity of $q^*$ at boundary must orthogonal to $x$-axis. Therefore $q^*$ can be extended to a symmetric periodic orbit by reflection about $x$-axis.

Any action minimizer corresponds to a solution of Stark problem which has no interior collision, thus it must be smooth in $(0,T)$. Suppose that there exists $t_0\in (0,T)$ such that $q^*(t_0)$ lies on the $x$-axis. If $\dot{q}^*(t_0)$ is tangent to $x$-axis, then the motion remains on $x$-axis, leading to a collision orbit, contradicting our assumption. Thus $\dot{q}^*(t_0)$ is transverse to $x$-axis. By reflecting the path in $[0,t_0]$, we get a new path with the same action as $q^*$.  This results in a non-smooth action minimizer. Contradiction! Therefore, $q^*$ cannot intersect $x$-axis in $(0,T)$.

 Recall that $q$ is in negative $x$-axis if and only if $\xi=0$, and $q$ is in positive $x$-axis if and only if $\eta=0$. By above argument, $q^*$ intersect negative and positive $x$-axis once each within one period. Denote $T_\tau$ to be the moment in $\tau$-coordinate for $t=T$, then 
 \begin{equation}\label{eq:Ttau}
 	2T_\tau=\frac{T_\xi}{2}=\frac{T_\eta}{2}.
 \end{equation}

The existence of this periodic orbit implies that $\xi$ and $\eta$ are both periodic but not constant. Then \eqref{eq:Ttau} contradicts to Proposition~\ref{pro:Txi>Teta}. In conclusion, our assumption that $q^*$ is a collision-free action minimizer is not valid. 

\textbf{Step 2:}  We have proved that the action minimizer $q^*$ has boundary collisions. Define a new path $\tilde{q}^*(t):=(-|q^*(t)|,0)$. 
Clearly, $\tilde{q}^*\in \Omega$ and all the three terms in action functional $\mathcal{A}$ for $\tilde{q}^*$ are less than those for $q^*$. Thus
\[\mathcal{A}(\tilde{q}^*)\le \mathcal{A}(q^*).\]
Since $q^*$ is an action minimizer. The equality holds, which implies that $q^*$ is always on negative $x$-axis. 

If both boundary have collision, then it's easy to decrease the action by perturbation at one boundary. Thus collision happens at only one boundary. At the other boundary, it necessary that the velocity of $q^*$ is $0$. Hence, $q^*$ is half of a collision ejection orbit in negative $x$-axis.
The proof is complete.
\qed

 \section{Dynamics on energy hypersurfaces}
Since the energy is preserved by motion, the fixed energy problem is concerned in many cases. Studying periodic orbits on a given energy hypersurface is crucial for understanding the dynamical behavior of a system. In \cite{Urs2023}, the authors characterized the bounded part of Stark problem's hypersurface under the critical value as the boundary of a concave toric domain. A.Takeuchi and L. Zhao \cite{Zhao2023} use a different approach to get sufficient conditions for energy hypersurface of Stark-type systems to be concave/convex toric domians, i.e. one needs to prove the monotonicity of $\frac{T_\xi}{T_\eta}$(as shown in Proposition~\ref{pro:4.1}).

Recall that our brake orbits in Theorem~\ref{thm:brake} satisfies $H=-2a\ge -2$. Thus they are not on compact  energy hypersurfaces. Although these hypersurface are not compact, the invariant tori exist for $-2\le H<0$. We will show that the dynamics on these invariant tori has some similarities with those for $H<-2$.

In $(\xi,\eta)$ coordinates, the system is regularized and integrable. When $H=h$ is fixed, the energy hypersurface is composed of two-dimensional invariant surfaces depending on the parameter $c$. Each invariant surface is a product surface of invariant curves for $\xi$ and $\eta$. Thus it's not hard to get the topology of the hypersurface. For $h\le 0$, $\eta$ is always periodic, and invariant curves for $(\xi,\xi')$ may have different types. For $h>0$, the invariant curves for $(\xi,\xi')$ are simple, while those for $(\eta,\eta')$ are more complicated.

Here we give an illustration for $h<-2$, which is the case studied in \cite{Urs2023}. When $c>1$,  the equation \eqref{eq:eta'} has no real solution for $(\eta,\eta')$, thus $c\in (-\infty,1]$. In this situation, as $c$ increase from $-\infty$ to $1$, the invariant curves for $(\xi,\xi')$ goes from $(h)$ to $(d)$ and then to $(c)$ in Fig.\ref{fig:xi}, while those for $(\eta,\eta')$ goes from $(e)$ to $(f)$ in Fig.\ref{fig:eta}. Clearly, $\eta$ is always periodic and degenerates to a point at $c=1$,  and $\xi$ always has two unbounded invariant curves for $c\in (-\infty,1]$. When $c\in [-1,1]$, a periodic curve appears and it degenerates to a point at $c=-1$. Thus the energy hypersurface has three components: two non-compact components and a compact one. A non-compact component is isomorphic to $\mathbb{R}^3$, which is composed of invariant cylinders and the cylinder degenerates to a line at $c=1$. The compact component is composed of invariant tori and the torus degenerates to a circle at $c=\pm 1$. In \cite{Urs2023}, it is proven that this compact component is the boundary of a concave toric domain, i.e., $\frac{T_\xi}{T_\eta}$ is strictly increasing for $c\in [-1,1]$.

Actually, these invariant tori exist when Fig.\ref{fig:xi_c} happens. For any $h<0$, this is restricted to the condition $c\in (-1,c_0)$, where $c_0=\min\{-1+\frac{h^2}{2},1\}$. The difference is that $c_0<1$ for $h\in (-2,0)$ and $\lim_{c\rightarrow c_0}T_\xi=+\infty$ for $h\in [-2,0)$. When $c=c_0$, the invariant tori break into the brake orbit and the sequence of orbits in $(v)$ of Theorem~\ref{thm:brake}.

In  previous sections, we've known that $T_\xi$ and $T_\eta$ are defined by $H$ and $c$. Now $H=h$ is fixed, thus both $T_\xi$ and $T_\eta$ are functions of $c$. In Proposition~\ref{pro:K} we've shown that $\frac{T_\xi}{T_\eta}$ is monotone with parameter $b$, in which the energy is not fixed. In the fixing energy case, there is a similar result proved in  \cite{Urs2023}.

\begin{proposition}[Proposition~4.1 in \cite{Urs2023}]\label{pro:4.1}
	$\frac{T_\xi}{T_\eta}$ is strictly increasing for $c\in (-1,c_0)$.
\end{proposition}

\begin{remark}
	 The Proposition~in \cite{Urs2023} is initially stated for $h<-2$, but it can be directly applied to the range $h\in [-2,0)$. It should be noted that the variable "c" in \cite{Urs2023} differs from the one used here. In \cite{Urs2023}, the author transforms various energy levels to $H=-\frac{1}{2}$ by adjusting $\epsilon$ (the external acceleration). Consequently, the invariant tori exist for $" c\in (0,\frac{1}{8\epsilon})"$ in \cite{Urs2023}, whereas in our paper, the invariant tori exist for $c\in (-1,\frac{h^2}{2}-1)$.
\end{remark}

\begin{proposition}\label{pro:4.2}
	For $h<0$, the following holds
	\begin{align}
		\lim_{c\rightarrow -1}\frac{T_\xi}{T_\eta}&=\left(1+\frac{4}{h^2}\right)^\frac{1}{4}\frac{K(0)}{K(\frac{h+\sqrt{h^2+4}}{2\sqrt{h^2+4}})}>1,  \label{eq:pro4.2.1} \\
	\lim_{c\rightarrow c_0}\frac{T_\xi}{T_\eta}&=\left\{\begin{aligned}
		&+\infty,& h\in [-2,0);\\&\Big(\frac{2h}{h-\sqrt{h^2-4}}\Big)^\frac{1}{2}\frac{K\Big(\frac{4}{(-h+\sqrt{h^2-4})^2}\Big)}{K(0)},& h<-2.     \label{eq:pro4.2.2}
	\end{aligned}\right.
\end{align}
\end{proposition}
\begin{proof}
	According to the Proposition~\ref{pro:T},  it is sufficient to prove the inequality in \eqref{eq:pro4.2.1}. Let $s=\frac{h+\sqrt{h^2+4}}{2\sqrt{h^2+4}}=\frac{1}{2}-\frac{1}{2}\sqrt{\frac{h^2}{h^2+4}}$, then $s\in (0,\frac{1}{2})$. It is equivalent to prove 
	\[K(s)(1-2s)^\frac{1}{2}<K(0).\]
	By definition of $K(s)$, we have
	\begin{align*}
		K(s)(1-2s)^\frac{1}{2}&=\int_{0}^{1}\frac{1}{\sqrt{1-t^2}}\sqrt{\frac{1-2s}{1-st^2}}dt,\\
		&<\int_{0}^{1}\frac{1}{\sqrt{1-t^2}}dt,\\
		&=\frac{\pi}{2}=K(0).
	\end{align*}
\end{proof}

The inequality in Proposition~\ref{pro:4.2} can also be viewed as a consequence of Proposition~\ref{pro:4.1} and the absence of elliptic-like orbits.

In \cite{Moser1970}, J.  Moser has shown that the Hamiltonian flow for Kepler problem with negative energy can embedded as the geodesic flow on the 2-sphere. The averaging method can be applied to the Kepler problem, and the Stark problem can be regarded as a perturbed Kepler problem. Considering the case $H=-\frac{1}{2}$ with the perturbation $\epsilon$ being small, \eqref{eq:seperation} becomes
\begin{equation}\label{eq:h-0.5}
	\frac{1}{2}(\xi^2+\xi'^2)+\frac{1}{2}(\eta^2+\eta'^2)+\frac{\epsilon}{2}(\eta^4-\xi^4)=2,
\end{equation}
where the Harmonic Oscillator part corresponds to the Hamiltonian for the regularized Kepler problem, and $\frac{\epsilon}{2}(\eta^4-\xi^4)$ is regarded as a perturbation. 

Up to reparametrisation of time by a constant factor, we can interpret the Hamiltonian flow of the Stark problem for field strenght $\epsilon$ and energy $h$ as the Hamiltonian flow of the Stark problem for field strength $c^2\epsilon$ and energy $ch$.Therefore, studying different energy levels for $\epsilon=1$ is equivalent to studying $H=-\frac{1}{2}$ for different $\epsilon$.  When $\epsilon$ is small in \eqref{eq:h-0.5}, this corresponds to $h\ll 0$ for $\epsilon=1$.

The negative energy hypersurface for Kepler problem  is a case where averaging method can be applied. Every non-degenerate critical point for the average of perturbation function will corresponds to periodic solution for perturbed system, with a period close to the unperturbed system. For Kepler problem, the number of critical points is as least $2$ given by Euler characteristic of $S^2$. Applying to $\eqref{eq:h-0.5}$ with $\epsilon$ small, we deduce that at least two periodic orbits of Stark problem will have periods close to those of the Kepler problem.

By Proposition~\ref{pro:4.2}, on these invariant tori $\frac{T_\xi}{T_\eta}>1$ belongs to a bounded interval for $h<-2$ and the length of interval tends to infinity as $h\rightarrow -2$. Numerical computation shows that $\frac{T_\xi}{T_\eta}$ is less than $2$ for $h<-2.01$. The interval becomes extremely narrow and $\frac{T_\xi}{T_\eta}>1$ is closed to $1$ as $h\ll -2$. Thus all periodic orbits on these invariant tori will move like in Fig.\ref{fig:elliptic} with very large period. The two collision orbits on the negative and positive $x$-axis are the only two orbits ensured by averaging method.

Note that we set $\epsilon=1$ for simplicity because systems for different $\epsilon$ are equivalent. We just remind that most of strange orbits take place in the energy interval $(-2\sqrt{\epsilon},0)$. For negative energy away from  $-2\sqrt{\epsilon}$, orbits move like in Fig.\ref{fig:elliptic}. As explained in introduction.

\section{The spatial case}\label{sec:spatial}
The equation for Stark problem in $\mathbb{R}^3$ is 
\begin{equation}\label{eq:stark3d}
	\left\{\begin{aligned}
		\ddot{x}&=-\frac{ x}{|q|^3},  \\
		\ddot{y}&=-\frac{ y}{|q|^3}, \\
		\ddot{z}&=-\frac{z}{|q|^3}+1.
	\end{aligned}
	\right.
\end{equation}
The Hamiltonian is 
\begin{equation*}
	H=\frac{1}{2}\dot{q}^2-\frac{1}{|q|}-z.
\end{equation*}

To separate the variables, a change of variables through the 3D parabolic coordinates is employed. The transformation formulas are the following:
\begin{equation}
	\left\{\begin{aligned}
		x&=\xi\eta\cos \phi,\\
		y&=\xi\eta\sin \phi,\\
		z&=\frac{1}{2}(\xi^2-\eta^2).
	\end{aligned}\right.
\end{equation}

For any $\phi$, $(\xi,\eta)\rightarrow (x,y,z)$ is generically a four to two map with two symmetric images about the z-axis. Let $dt=(\xi^2+\eta^2)d\tau$ and $\xi'=\frac{d\xi}{d\tau}, \eta'=\frac{d\eta}{d\tau}$.  Since the system is invariant under the rotation about $z$-axis, the component of angular momentum along the z-axis is a first integral. 
\begin{equation*}
	L=x\dot{y}-y\dot{x}=\xi^2\eta^2\dot{\phi}.
\end{equation*}
The Hamiltonian can be written as 
\begin{equation}\label{eq:Hin3d}
	H=\frac{1}{2}\frac{\xi'^2+\eta'^2}{\xi^2+\eta^2}+\frac{1}{2}\frac{L^2}{\xi^2\eta^2}-\frac{2}{\xi^2+\eta^2}-\frac{\xi^2-\eta^2}{2}.
\end{equation}
Similar to \eqref{eq:seperation}, we have
\begin{equation}\label{eq:sep3d}
	H\xi^2-\frac{1}{2}\xi'^2+1+\frac{1}{2}\xi^4-\frac{L^2}{2\xi^2}=-H\eta^2+\frac{1}{2}\eta'^2-1+\frac{1}{2}\eta^4+\frac{L^2}{2\eta^2}=-c.
\end{equation}

Hence, the system is also separable for $\xi,\eta$. Once $\xi,\eta$ is solved for given $L$, then $\phi$ can be determined by $\phi(t)=\phi(0)+\int_{0}^{t}\frac{L}{\xi^2\eta^2}ds$. To obtain a periodic orbit for the spatial Stark problem, it's necessary that both $\xi $ and $\eta$ are periodic.   As in the planar case, we study the invariant curves for $\xi,\eta$. 

Note that $L=0$ if and only if $(\dot{q}\times q) \cdot (0,0,1)=0$. The three vectors $\dot{q},q,(0,0,1)$ are on a plane, which restricts to the planar case. Thus we assume that $L\ne 0$. The invariant curves for $(\xi,\xi')$ are essentially defined by the graph of following cubic polynomial in first quadrant. 
\begin{equation}
	f(u)=u^3+2hu^2+2(c+1)u-L^2.
\end{equation}

We could also enumerate all possible types of invariant curves for $(\xi,\xi')$ as in Fig.\ref{fig:xi}. Here we only consider the cases in which $\xi$ could be periodic. It's easy to see that a necessary condition is that $f(u)=0$ has three real roots(some of which may be identical). Let $u_1\le u_2\le u_3$ be three roots of $f(u)=0$. Since $u_1u_2u_3=L^2> 0$, either $u_1\le u_2<0<u_3$ or $0<u_1\le u_2\le u_3$. When $u_1\le u_2<0<u_3$, the graph of $f(s)$ in first quadrant is unbounded. We only need to consider the case $0<u_1\le u_2\le u_3$. All possible periodic invariant curves for $(\xi,\xi')$ are listed in Fig.\ref{fig:3d_xi}, while for other cases $(\xi,\xi')$ is unbounded.

\begin{figure}
	\begin{subfigure}{0.23\textwidth}
		\includegraphics[width=\linewidth]{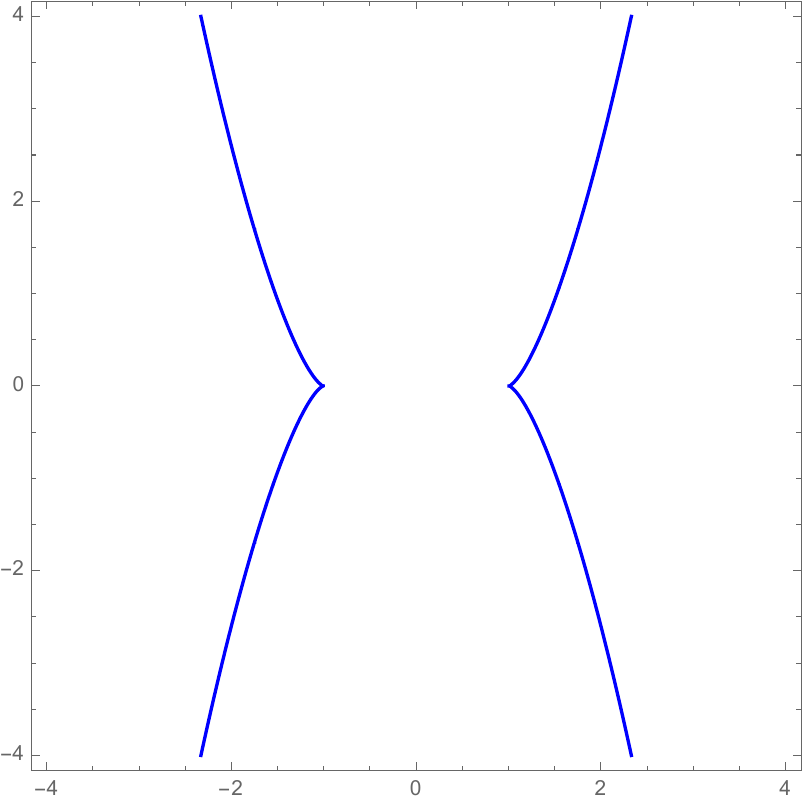}
		\caption{$u_1=u_2=u_3$}
		\label{fig:3d_xi1}
	\end{subfigure}
	\hfill
	\begin{subfigure}{0.23\textwidth}
		\includegraphics[width=\linewidth]{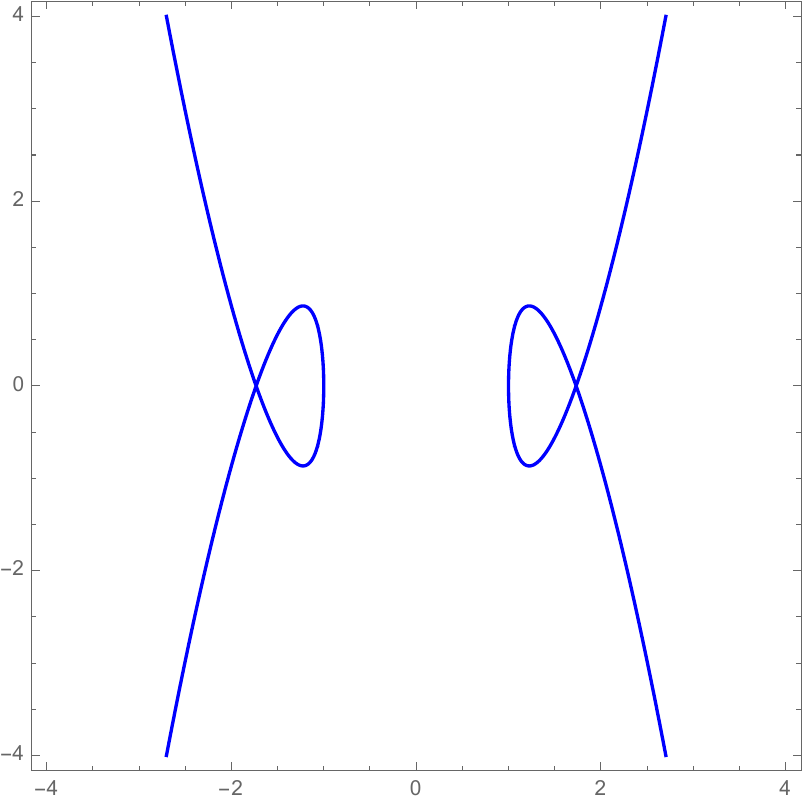}
		\caption{$u_1<u_2=u_3$}
		\label{fig:3d_xi2}
	\end{subfigure}
	\hfill
	\begin{subfigure}{0.23\textwidth}
		\includegraphics[width=\linewidth]{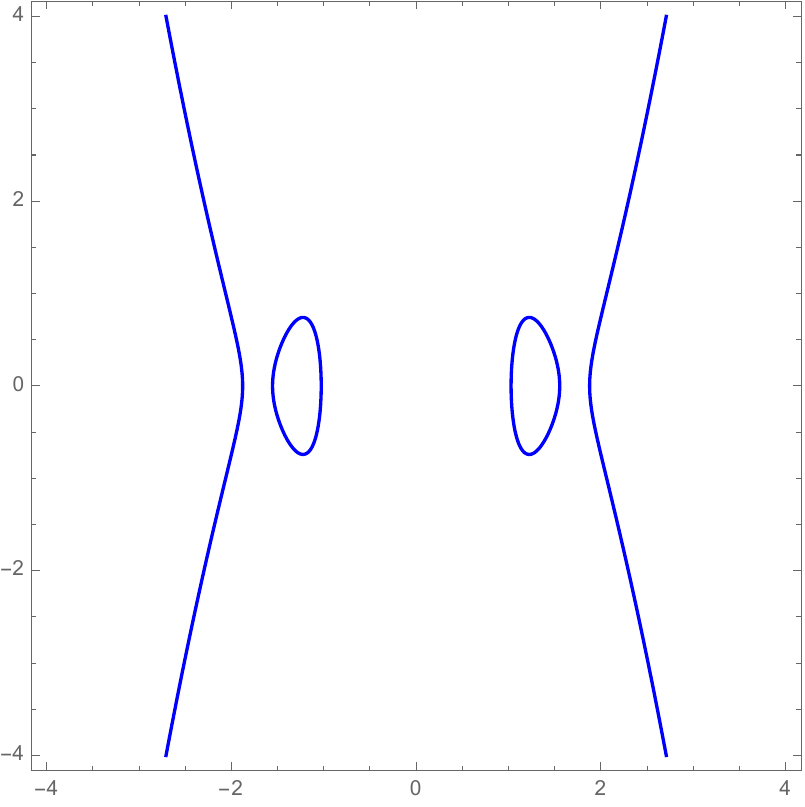}
		\caption{$u_1<u_2<u_3$}
		\label{fig:3d_xi3}
	\end{subfigure}
	\hfill
	\begin{subfigure}{0.23\textwidth}
		\includegraphics[width=\linewidth]{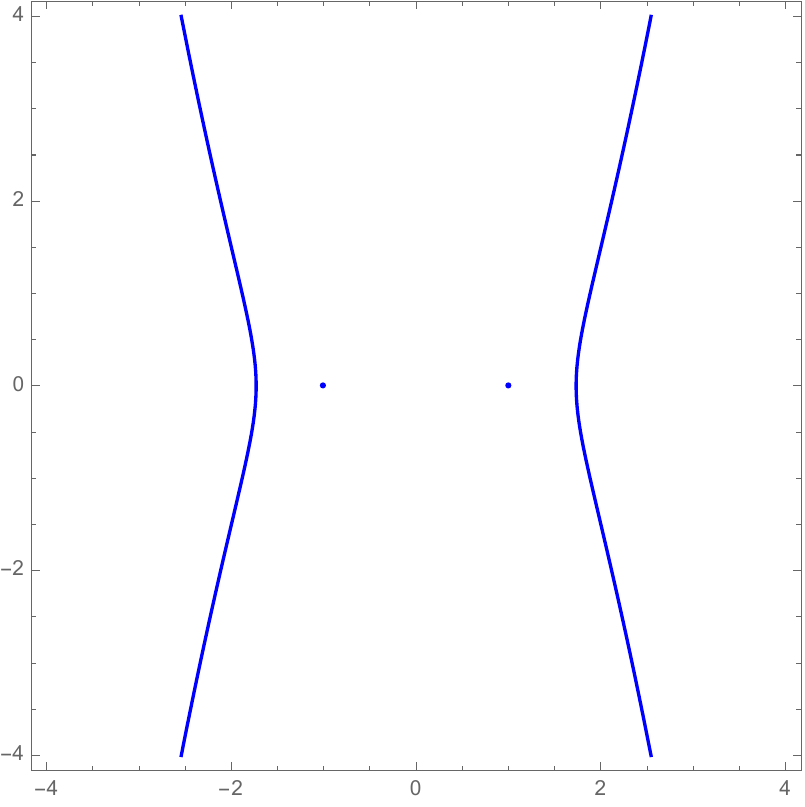}
		\caption{$u_1=u_2<u_3$}
		\label{fig:3d_xi4}
	\end{subfigure}
	\caption{All possible types of invariant curves for $(\xi,\xi')$ in which $\xi$ could be periodic.}
	\label{fig:3d_xi}
\end{figure}
Similarly, The invariant curve for $(\eta,\eta')$ is determined by the graph of the corresponding cubic polynomial in first quadrant. 
\begin{equation}
	g(v)=-v^3+2hv^3-2(c-1)v-L^2.
\end{equation}
Since $v_1v_2v_3=-L^2<0$, the graph of $g(v)$ is non-empty in first quadrant if and only if there are three real roots satisfying $v_1<0<v_2\le v_3$. Thus $(\eta,\eta')$ is always periodic and the possible invariant curves are listed in Fig.\ref{fig:3d_eta}.

\begin{figure}
	\begin{subfigure}{0.23\textwidth}
		\includegraphics[width=\linewidth]{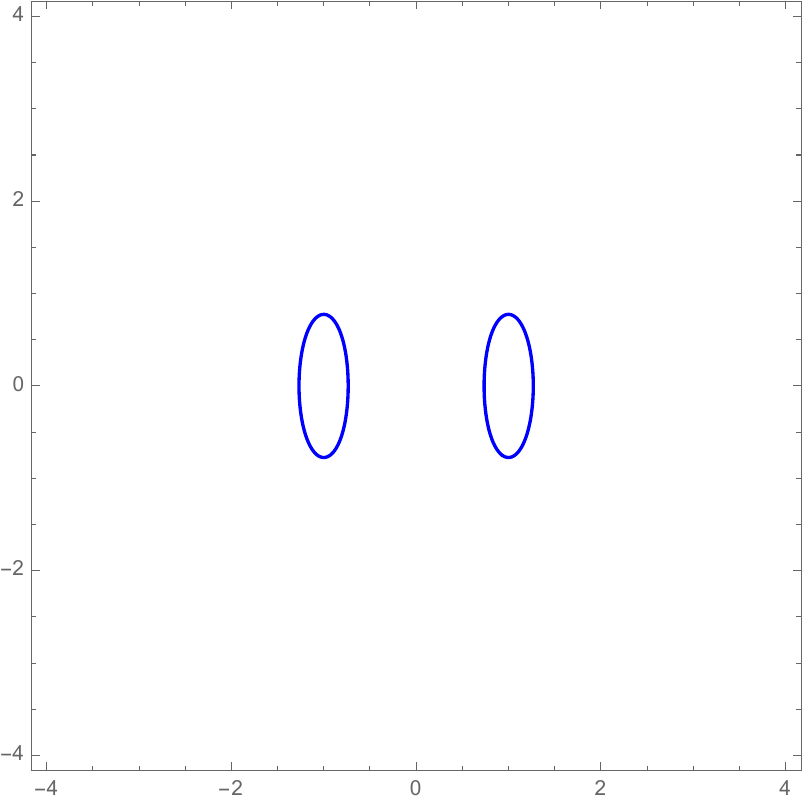}
		\caption{$v_1<0<v_2<v_3$}
		\label{fig:3d_eta1}
	\end{subfigure}
\ \ 
	\begin{subfigure}{0.23\textwidth}
		\includegraphics[width=\linewidth]{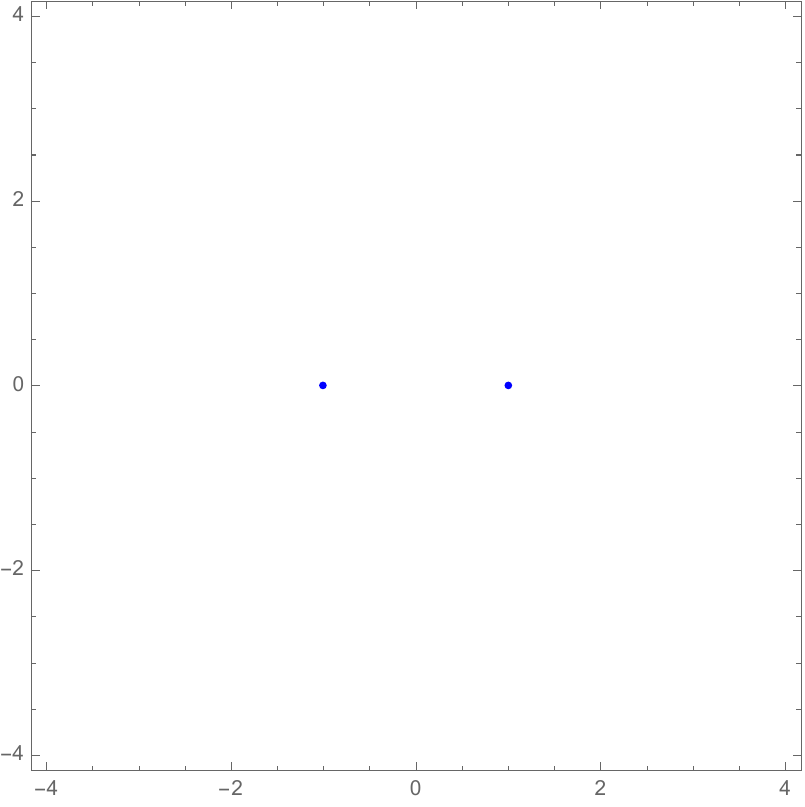}
		\caption{$v_1<0<v_2=v_3$}
		\label{fig:3d_eta2}
	\end{subfigure}
	
	\hfill
	
	\caption{All possible types of invariant curves for $(\eta,\eta')$ in which $\eta$ could be periodic.}
	\label{fig:3d_eta}
\end{figure}

Fig.\ref{fig:3d_xi} and Fig.\ref{fig:3d_eta} correspond to $(b)-(e)$ of Fig.\ref{fig:xi} and $(e)-(f)$ of  Fig.\ref{fig:eta}. To transition from the planar case to the spatial case, we simply split the image of the planar case into two parts from the middle. This should be the effect caused by nonzero angular momentum.

To study the periodic orbits, it suffices to consider the periodic invariant curve for $\xi$ and $\eta$. They have the following types for $(\xi,\xi')$. In Fig.\ref{fig:3d_xi1}, there is a degenerate fixed point. The fixed points in Fig.\ref{fig:3d_xi2} and Fig.\ref{fig:3d_xi4} are hyperbolic and elliptic, respectively. In Fig.\ref{fig:3d_xi3}, $(\xi,\xi')$ is non-constant periodic. For other cases, $\xi$ is necessary unbounded. Invariant curves for $(\eta,\eta')$ are much simpler, with only two admissible cases listed in Fig.\ref{fig:3d_eta}. As in the planar case, all bounded solutions for spatial Stark problem can be classified from Fig.\ref{fig:3d_xi} and Fig.\ref{fig:3d_eta}.

When both $\xi$ and $\eta$ are fixed points, it corresponds to a uniform circular motion in initial coordinates. Typically, the motion will oscillate in both the $z$-direction and the $xy$-plane. If both  $\xi,\eta$ are periodic, we get an invariant torus and the motion for $\xi,\eta$ is periodic if and only if their period $T_\xi,T_\eta$is rationally dependent.

It should be remind that the rationally dependent of $T_\xi$ and $T_\eta$ does not imply a periodic orbit in the initial coordinate system, since $\phi$ has not been taken into account. $\phi$ may differ by a rotation around the $z$-axis after a period. A periodic orbit require the change in $\phi$ over one period is a rational multiple of $2\pi$. Otherwise, it would be a quasi-periodic orbit in space. Because the system is invariant under rotation about $z$-axis. We will neglect $\phi$ and focus on periodic solutions in $(\xi,\eta)$ coordinates.

We need to investigate how the type of invariant curve changes with the parameters $H=h,c,L$. 

When Fig.\ref{fig:3d_xi1} happens, it holds
\[h=-\frac{3}{2}L^\frac{2}{3},\ \ \ \  c=\frac{3}{2}L^\frac{4}{3}-1.\]

When Fig.\ref{fig:3d_xi2} happens, then 
\[f(u)=(u-u_1)(u-u_2)^2=u^3+2hu^2+2(c+1)u-L^2.\]
It holds
\begin{equation}\label{eq:Hc1}
	\left\{\begin{aligned}
		&2u_2+\frac{L^2}{u_2^2}=-2h,\\
		&(u_2+\frac{2L^2}{u_2^2})u_2=2(c+1),\\
		& u_2>L^{\frac{2}{3}}.
	\end{aligned}\right.
\end{equation}
Since $u_2>L^{\frac{2}{3}}$, there is a bijection from $u_2\in (L^{\frac{2}{3}},+\infty)$ to $H\in (-\infty,-\frac{3}{2}L^{\frac{2}{3}})$. Thus by \eqref{eq:Hc1}, $c$ is a function of $H$. We denote this function by $C_{2,\xi}(h)$.

When Fig.\ref{fig:3d_xi4} happens, we can similarly get a function from $H\in  (-\infty,-\frac{3}{2}L^{\frac{2}{3}})$ to $u_1\in (0,L^\frac{2}{3})$ then to $c$. We denote this function by $C_{1,\xi}(h)$.

We can extend the definition of $C_{1,\xi}$ and $C_{2,\xi}$ to the limit case $h=-\frac{3}{2}L^\frac{2}{3}$. Then the following holds.
\begin{lemma}\label{lem:C1C2}
	For any $h\in (-\infty,-\frac{3}{2}L^{\frac{2}{3}}]$, it holds $C_{1,\xi}(h)\le C_{2,\xi}(h)$. The equality holds if and only if $h=-\frac{3}{2}L^\frac{2}{3}$.
\end{lemma}
\begin{proof}
	It suffices to prove that the inequality holds for $h\in  (-\infty,-\frac{3}{2}L^{\frac{2}{3}})$.
	Consider two functions $f_1(s)=2s+\frac{L^2}{s^2}$ and $f_2(s)=(s+\frac{2L^2}{s^2})s$, where $s\in(0,\infty)$. Then 
	\begin{equation}
		f_2'(s)=2s-\frac{2L^2}{s^2}=s(2-\frac{2L^2}{s^{3}})=s{f_1'(s)}.
	\end{equation}

Both $f_1(s)$ and $f_2(s)$ attain its minimum at $s=L^{\frac{2}{3}}$. For any $h\in  (-\infty,-\frac{3}{2}L^{\frac{2}{3}})$, there exists two positive numbers $s_1\in (0,L^{\frac{2}{3}})$ and $s_2\in (L^{\frac{2}{3}},+\infty)$ such that $f_1(s_1)=f_1(s_2)=-2h$. By our definition, we have $C_{1,\xi}(h)=\frac{f_2(s_1)}{2}-1, C_{2,\xi}(h)=\frac{f_2(s_2)}{2}-1$. It suffices to prove $f_2(s_1)<f_2(s_2)$.

    Note that $f_1',f_2'$ are negative in $(0,L^\frac{2}{3})$ and positive in $(L^\frac{2}{3},+\infty)$.
\begin{equation}\label{eq:f1}
	-2h=3L^{\frac{2}{3}}+\int_{L^\frac{2}{3}}^{s_2}f_1'(s)ds=3L^{\frac{2}{3}}-\int_{s_1}^{L^\frac{2}{3}}f_1'(s)ds.
\end{equation}
Hence,
\begin{equation}
\begin{aligned}\label{eq:f2}
	f_2(s_1)&=f_2(L^\frac{2}{3})-\int_{s_1}^{L^\frac{2}{3}}f_2'(s)ds
	=3L^\frac{4}{3}-\int_{s_1}^{L^\frac{2}{3}}sf_1'(s)ds,\\
	&<3L^\frac{4}{3}-L^\frac{2}{3}\int_{s_1}^{L^\frac{2}{3}}f_1'(s)ds
	=-2hL^\frac{2}{3},\\
	&<3L^\frac{4}{3}+\int_{L^\frac{2}{3}}^{s_2}sf_1'(s)ds
	=3L^\frac{4}{3}+\int_{L^\frac{2}{3}}^{s_2}f_2'(s)ds,\\
	&=f_2(s_2).
\end{aligned}
\end{equation}
We actually have
\begin{equation}
	-1<C_{1,\xi}(h)<-hL^\frac{2}{3}-1<C_{2,\xi}(h).
\end{equation}
\end{proof}

By Fig.\ref{fig:3d_xi} and Lemma~\ref{lem:C1C2}, we have proved the following.
\begin{proposition}
	Given $L\ne 0$ and $h\in (-\infty,-\frac{3}{2}L^\frac{2}{3})$, Fig.\ref{fig:3d_xi3} happens if and only if $c\in (C_{1,\xi}(h),C_{2,\xi}(h))$.
\end{proposition}

When Fig.\ref{fig:3d_eta2} happens, we have
\[-g(v)=(v-v_1)(v-v_2)^2=v^3-2hv^3+2(c-1)v+L^2.\]
Thus,
\begin{equation}\label{eq:Hceta}
	\left\{\begin{aligned}
		&2v_2-\frac{L^2}{v_2^2}=2h,\\
		&(v_2-\frac{2L^2}{v_2^2})v_2=2(c-1).
	\end{aligned}\right.
\end{equation}
Note that $v_2>0$ is uniquely defined by $h$, thus $c$ is a function of $h$, which we denote by $C_\eta(h)$. One can verifies that Fig.\ref{fig:3d_eta1} happens when $c<C_\eta(h)<1$,  and $c>C_\eta(h)$ is not admissible because the graph of $g(v)$ is empty in the first quadrant.

\begin{lemma}
For the equation (\ref{eq:sep3d}),  periodic solutions $(\xi,\eta)$ exist if and only if $L^2\le \big(\frac{16}{27}\big)^\frac{3}{2}.$
\end{lemma}
\begin{proof}
	From \eqref{eq:f1} and \eqref{eq:f2}, one deduces that both $C_{1,\xi}(h)$ and $C_{2,\xi}(h)$ are strictly decreasing functions of $h\in (-\infty,-\frac{3}{2}L^\frac{2}{3}]$.  The equation \eqref{eq:Hceta} implies that $C_\eta(h)$ is strictly increasing function of $h\in \mathbb{R}$. By our argument about $\xi$, $(\xi,\xi')$ could be periodic only in $h\in (-\infty,-\frac{3}{2}L^\frac{2}{3}]$ and $c\in [C_{1,\xi}(h),C_{2,\xi}(h)]$. Recall the admissible  $c$ for $\eta$ is $c\le C_\eta(h)$. Hence the necessary condition for existence of periodic orbit is $C_{1,\xi}(-\frac{3}{2}L^\frac{2}{3})\le C_\eta(-\frac{3}{2}L^\frac{2}{3})$.
	By definitions of $C_{1,\xi}(h)$ and $C_\eta(h)$, this inequality  is equivalent to 
	\begin{equation*}
		\frac{3}{2}L^\frac{4}{3}-1\le -\frac{15}{8}L^\frac{4}{3}+1.
	\end{equation*}
And then,
    \begin{equation*}
    	L^2\le\Big(\frac{16}{27}\Big)^\frac{3}{2}.
    \end{equation*}
    
    When $C_{1,\xi}(-\frac{3}{2}L^\frac{2}{3})\le C_\eta(-\frac{3}{2}L^\frac{2}{3})$, then there exist  $h\le -\frac{3}{2}L^\frac{2}{3}$ and $c\le C_\eta(h)$ such that $\xi$ has a fixed point, which is clearly a periodic solution for $(\xi,\eta)$. Thus this condition is also sufficient.
\end{proof}

Given $L$ and $h$, the solutions $(\xi,\eta)$  for (\ref{eq:sep3d}) can be determined by the parameter $c$.   For $L^2< \big(\frac{16}{27}\big)^\frac{3}{2}$,  an illustration of $C_{1,\xi}(h),C_{2,\xi}(h)$ and $C_\eta(h)$ is given in Fig.\ref{fig:Ch}, 
and let $h_1,h_2$ be the unique point such that $C_{1,\xi}(h_1)=C_\eta(h_1)$ and $C_{2,\xi}(h_2)=C_\eta(h_2)$, respectively. Then $h_1=h_2$ for $L^2= \big(\frac{16}{27}\big)^\frac{3}{2}$ and $h_1<h_2$ for $L^2< \big(\frac{16}{27}\big)^\frac{3}{2}$.

\begin{figure}
	\begin{center}
			\includegraphics[width=0.8\linewidth]{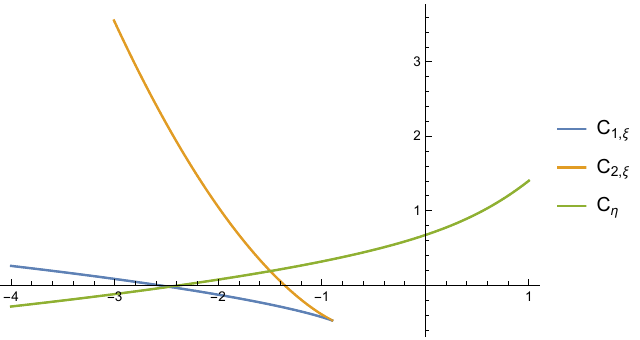}
		\caption{An illustration of $C_{1,\xi}(h),C_{2,\xi}(h)$ and $C_\eta(h)$ for $L=0.45$.}
		\label{fig:Ch}
	\end{center}
	
\end{figure}

Combining Fig.\ref{fig:3d_xi}, Fig.\ref{fig:3d_eta} and Fig.\ref{fig:Ch}, the conclusions of this section can be summarized in the following theorem:

\begin{theorem}\label{thm:spatial}
Given the angular momentum $L$, energy $h$,  constant $c$,  and defined $h_1$, $h_2$,  $C_{1\xi}(h)$, $C_{2,\xi}(h)$ and $C_\eta(h)$ as above (if they exist).  The orbits for the spatial Stark problem (\ref{eq:stark3d}) can be classified into the following cases:\\
	(a) When $L^2> \big(\frac{16}{27}\big)^\frac{3}{2}$ or $h>-\frac{3}{2}L^\frac{2}{3}$, all orbits are unbounded;\\
	(b) When $L^2= \big(\frac{16}{27}\big)^\frac{3}{2}$ and $h=-\frac{3}{2}L^\frac{2}{3}$,
there is a unique bounded orbit for $c=-\frac{1}{9}$, and it is a circular orbit. \\
	(c) When $0<L^2< \big(\frac{16}{27}\big)^\frac{3}{2}$, we have $h_1<h_2<-\frac{3}{2}L^\frac{2}{3}$. The types of bounded orbits can be divided into following subcases depending on both $h$ and $c$, while all other orbits are unbounded.
	\begin{enumerate}
		\item [$(c1)$] $h\in (h_2,-\frac{3}{2}L^\frac{2}{3})$,  bounded orbits exist only if $c\in [C_{1\xi}(h),C_{2,\xi}(h)]$, and they are typically oscillatory orbits in space.  $\eta$ is always non-constantly periodic, thus there is no circular orbit. 
		\item [$(c2)$] $h=h_2$, there exist a unstable circular orbit for $c=C_{2,\xi}(h)$ and a orbit that tends to the circular orbit as $t \rightarrow \pm\infty$. The circular orbit corresponding to Fig.\ref{fig:3d_eta2} and the hyperbolic fixed point in Fig.\ref{fig:3d_xi2};
		\item [$(c3)$] $h\in (h_1,h_2)$, the energy surface for $\xi,\eta$ has three connected components, one of which is compact. Just like $h<-2$ in planar case, the compact energy surface is composed of invariant tori given by parameters $c\in [C_{1,\xi}(h),C_\eta(h)]$. These invariant tori degenerate into circles for $c$ at the endpoints of interval;
		\item  [$(c4)$] $h=h_1$, there exists a stable circular orbit  for $c=C_{1,\xi}(h)$, which corresponds to Fig.\ref{fig:3d_eta2} and the elliptic fixed point in Fig.\ref{fig:3d_xi4}.
	\end{enumerate}
\end{theorem}

By Theorem~\ref{thm:spatial}, for $0<L^2<\big(\frac{16}{27}\big)^\frac{3}{2}$, there exist two circular orbits with different energy. The one with smaller energy is stable, and the other is unstable. The energy of two circular orbits gets closer to each other as $L^2$ gets closer to $\big(\frac{16}{27}\big)^\frac{3}{2}$. They eventually become a degenerate circular orbit when $L^2=\big(\frac{16}{27}\big)^\frac{3}{2}$.

For circular orbits, the $z$-axis component of the central attractive force cancels out with the external acceleration, while the $xy$-plane component provides the centripetal force for circular motion. Obviously, for any $z\in (0,1)$, there exists exactly one circular orbit. Let $r=\sqrt{x^2+y^2}$ be the radial of motion. Then $h,r,L$ can be seen as function of $z$. By direct computation, we have:
\begin{equation}\label{eq:rLh}
	\left\{\begin{aligned}
		r(z)&=\sqrt{z^\frac{2}{3}-z^2},\\
		L(z)&=z^\frac{1}{6}-z^\frac{3}{2},\\
		h(z)&=-\frac{z^{-\frac{1}{3}}}{2}-\frac{3}{2}z.
	\end{aligned}\right.
\end{equation}

The graphs of these three functions are shown in Fig.\ref{fig:Lhr}, where 
the graph of $r(z)$ is the same as the curve $\ddot{x}=0$ for planar case. 
Both $L$ and $h$ obtain their maximum at $z=\big(\frac{1}{3}\big)^\frac{3}{2}$, which corresponds to a circular orbit defined by the degenerate fixed point in Fig.\ref{fig:3d_xi1}. 

For $L^2<\big(\frac{16}{27}\big)^\frac{3}{2}$, there exist two circular orbits with angular momentum $L$,  energy $h_1$ and $h_2$, whose $z$-coordinates located on either side of $\big(\frac{1}{3}\big)^\frac{3}{2}$. Clearly, the orbit with smaller $z$-coordinate has smaller energy $h_1$. It is easy to verify that in this case both $\xi$ and $\eta$ are non-degenerate elliptic fixed point. Thus we have proved the following.
\begin{proposition}\label{pro:stable}
	For any $s\in (0,1)$, there exists exactly one circular orbit with $z=s$. With $L$ being fixed, the circular orbit is stable if and only if $s<\big(\frac{1}{3}\big)^\frac{3}{2}$.
\end{proposition}
\begin{figure}
	\begin{center}
		\includegraphics[width=0.8\linewidth]{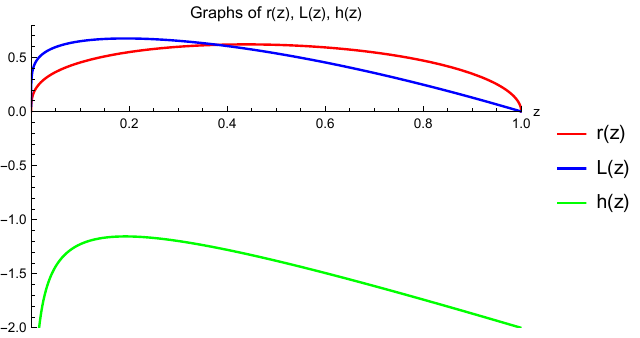}
		\caption{The graphs of $r(z),L(z)$ and $h(z)$.}
		\label{fig:Lhr}
	\end{center}
\end{figure}

We also studies the dynamics on these invariant tori as in planar case.  By \eqref{eq:sep3d}, 
\begin{align*}
	\frac{d\tau}{d \xi}&=\frac{1}{\sqrt{\xi^4+2H\xi^2+2(c+1)-\frac{L^2}{\xi^2}}}, \\
	\frac{d\tau}{d \eta}&=\frac{1}{\sqrt{-\eta^4+2H\eta^2-2(c-1)-\frac{L^2}{\eta^2}}}.
\end{align*}
Let $0<\xi_1^2<\xi_2^2<\xi_3^2$ and $-\eta_1^2<0<\eta_2^2<\eta_3^2$ be the zero points of corresponding polynomials in denominators,  we have
\begin{align}
	T_\xi&=2\int_{\xi_1}^{\xi_2}\frac{d\xi}{\sqrt{\xi^4+2H\xi^2+2(c+1)-\frac{L^2}{\xi^2}}}, \label{eq:Txi3d}\\
	T_\eta&=2\int_{\eta_1}^{\eta_2}\frac{d\tau}{\sqrt{-\eta^4+2H\eta^2-2(c-1)-\frac{L^2}{\eta^2}}}. \label{eq:Teta3d}
\end{align}
Then,
\begin{align*}
	T_\xi&=2\int_{\xi_1}^{\xi_2}\frac{d\xi}{\sqrt{\xi^4+2H\xi^2+2(c+1)-\frac{L^2}{\xi^2}}},\\
	&=2\int_{\xi_1}^{\xi_2}\frac{\xi d\xi}{\sqrt{(\xi^2-\xi_1^2)(\xi^2-\xi_2^2)(\xi^2-\xi_3^2)}},\\
	&=2\int_{Z_1}^{Z_2}\frac{dZ}{\sqrt{\Big((\xi_3^2-\xi_1^2)-Z^2\Big)\Big(Z^2-(\xi_3^2-\xi_2^2)\Big)}},
\end{align*}
where $Z_1^2=\xi_3^2-\xi_2^2$ and $Z_2^2=\xi_3^2-\xi_1^2$. 

By well known results about elliptic integral, we obtain
\begin{equation}\label{eq:3dTxi}
	T_\xi=\frac{2K\Big(\frac{\xi_2^2-\xi_1^2}{\xi_3^2-\xi_1^2}\Big)}{\sqrt{\xi_3^2-\xi_1^2}}.
\end{equation}
Similar computation yields
\begin{equation}\label{eq:3dTeta}
	T_\eta=\frac{2K\Big(\frac{\eta_3^2-\eta_2^2}{\eta_1^2+\eta_3^2}\Big)}{\sqrt{\eta_1^2+\eta_3^2}}.
\end{equation}

By letting $\xi_1=0, \eta_2=0$ in \eqref{eq:3dTxi} and \eqref{eq:3dTeta}, we obtain \eqref{eq:Txi} and \eqref{eq:Teta} in the planar case. We conjecture $T_\xi$ and $T_\eta$ have the same property as in Proposition~\ref{pro:4.1} for $c\in (C_{1,\xi}(h),C_\eta(h))$, whenever these invariant tori exist for given $L$ and $h$. This question seems to be more complicated than Proposition~\ref{pro:4.1} since we have to deal with cubic polynomial and their roots. 

We also believe that all the bounded orbits must stay in the unit sphere all the time.
It suffices to prove that $\xi_2^2+\eta_3^2\le 2$. 
If the admissible $c$ for bounded orbits in the spatial case is a subset of those in the planar case, then the proof will be easy. Because the additional term $\frac{L^2}{\xi^2}$ will produce a smallest positive root, while decreasing the existing two roots. However, the admissible $c$ depends on $h$ and $L$ and is not a subset of those in the planar case. Due to the complexity of the computations, we stop our investigation of the Stark problem here.

Although the spatial Stark problem has some similarity with the planar case,  there are also many differences. Some interesting phenomena have appeared in the spatial case. Here, we make a comparison between them. 
\begin{enumerate}
	\item In planar case, all quasi-periodic orbits will eventually go close to the origin. In spatial case, due to the nonzero angular momentum, all quasi-periodic orbits will stay away from the origin. This is because both $\xi$ and $\eta$ have positive lower bound.
	\item Circle orbits can only exist in the spatial case. In the planar case, all bounded orbits for $(\xi,\eta)$ are in fact on invariant tori(or a circle), because $\xi$ and $\eta$ can not be fixed point at the same time. In spatial case, they can be fixed point at the same time, which corresponds to circle orbits.
	\item There exist Lyapunov stable orbits in spatial case as shown in Proposition~\ref{pro:stable}. The orbit corresponds to the left vertex of the triangle in Fig.\ref{fig:Ch}, where the admissible set $c<C_\eta(h)$ intersects the interval $[C_{1,\xi}(h),C_{2,\xi}(h)]$ for bounded orbit at a single point. Since the orbit can not jump from bounded case to unbounded, after small perturbation, the admissible parameters for $L,h,c$ should closed to the initial point. Thus perturbed orbits should be close to the circular orbit.
	\item Since $L\ne 0$, admissible $h,c$ for bounded orbits is  bounded, which is a triangular region formed by the intersection of three curves $C_{1,\xi},C_{2,\xi},C_\eta$.  As $L^2\rightarrow \big(\frac{16}{27}\big)^\frac{3}{4}$, the triangular region for bounded orbits goes to a single point. As $L\rightarrow 0$, $C_{1,\xi}\rightarrow -1$,$C_{2,\xi}\rightarrow c_0$ and $C_\eta\rightarrow 1$. For the planar case, admissible $h,c$ for bounded orbits is the region bounded by $c_0(h),c=+1$ and $c=-1$, which is the limit case for $L\rightarrow 0$.

\end{enumerate}

\textbf{Acknowledgements:} The authors would like to thanks Kuochang Chen for encouraging us to study the problem, Zhihong Xia for useful discussions.

\bibliographystyle{abbrvnat}
\bibliography{reference.bib}

\end{document}